\documentclass[11pt, a4paper, reqno]{amsart} 
\usepackage[left=2.5cm, right=2.5cm, bottom=3cm]{geometry}
\usepackage[utf8]{inputenc}
\usepackage{epigraph}
\usepackage{bbm}
\usepackage{setspace}
\usepackage{epigraph}
\usepackage{csquotes}
\usepackage{hyperref}
\usepackage[backend=biber,style=numeric,bibencoding=utf8,isbn=false,url=false,sorting=nty,doi=false,giveninits=true,maxbibnames=10,maxalphanames=10]{biblatex}
\usepackage[english]{babel} 
\usepackage{amsmath, amsfonts,amssymb}
\allowdisplaybreaks
\usepackage{enumitem}
\usepackage{mathtools}
\usepackage{url}
\usepackage{tikz-cd}
\usepackage{array,esint}
\usepackage{tikz}
\usepackage{tikz-3dplot}
\usepackage{pgfplots}
\pgfplotsset{compat=1.18}
\usepgfplotslibrary{fillbetween}
\usetikzlibrary{intersections}
\usetikzlibrary{arrows,decorations.markings}
\usetikzlibrary{decorations.text}
\usetikzlibrary{decorations.footprints}
\usetikzlibrary{decorations.fractals}
\usetikzlibrary{shapes.geometric}

\AtBeginBibliography{\small}
\renewbibmacro{in:}{}  

\newtheorem{theorem}{Theorem}[section]
\newtheorem{proposition}[theorem]{Proposition}
\newtheorem{corollary}[theorem]{Corollary}
\newtheorem{lemma}[theorem]{Lemma}
\newtheorem{theoremA}{Theorem}

\newtheorem{conjecture}[theorem]{Conjecture}

\theoremstyle{definition}
\newtheorem{definition}[theorem]{Definition}

\theoremstyle{definition}
\newtheorem{remark}[theorem]{Remark}

\newcommand{\R}{\mathbb{R}}
\newcommand{\C}{\mathbb{C}}
\newcommand{\N}{\mathbb{N}}
\newcommand{\Z}{\mathbb{Z}}

\newcommand{\T}{\mathbb{T}}

\newcommand{\bx}{\mathbf{x}}
\newcommand{\by}{\mathbf{y}}

\newcommand*\diff{\mathop{}\!\mathrm{d}}

\title{On the regularity of Fourier interpolation formulas}

\author{Gabriele Cassese}
\address{Mathematical Institute, Andrew Wiles Building, University of Oxford. OX2 6GG - Oxford, United Kingdom} 
\email{gabriele.cassese@maths.ox.ac.uk}
\author{Jo\~ao P. G. Ramos}
\address{Department of Mathematics, Faculty of Sciences of the University of Lisbon. PT-1749-016 - Lisbon, Portugal}
\email{joaopgramos95@gmail.com}

\begin{document}

\begin{abstract} By applying new functional analysis tools in the framework of Fourier interpolation
formulas, such as sc-Fredholm operators and Schauder frames, we are able to improve and refine
several properties of these aforementioned formulas on the real line.

As two examples of our main contributions, we highlight: (i) that we may upgrade perturbed
interpolation bases all the way to the \textit{Schwartz space}, which shows that even the perturbed
interpolation formulas are as regular as the Radchenko–Viazovska case; (ii) that a certain subset
of the interpolation formulae considered by Kulikov–Nazarov–Sodin may actually be upgraded
to be convergent in the Schwartz class, giving a first partial answer to a question posed by the
authors of \cite{KNS}.

As a final contribution of this work, we also show that, if the perturbations are sufficiently
tiny, then even \textit{analyticity properties} of the basis functions are preserved. This shows, in particular, that any function that vanishes on all but finitely many of the (perturbed) nodes is
\textit{automatically analytic}, a feature previously only known to hold in supercritical contexts \cite{RS1, KNS} besides the Radchenko–Viazovska case \cite{RV}.
\end{abstract}

\maketitle

\section{Introduction}

\subsection{Context and background} 
The 
problem of reconstructing
a function from some of its values 
and the ones of its Fourier transform
is one that has been a very active
area of research for a long time,
in harmonic analysis as well as in
applied mathematics, being one of the pillars of modern \emph{signal processing theory}, with consequences not only in the field of signal processing, but also in quantum mechanics, x-ray crystallography, and several other areas - see, for instance, \cite{reichenbach,seelamantula,ismagilov} for some such instances. 

In mathematical terms, the problem can be formulated as follows: find two sets, $A$ and $B$, respectively in time and frequency space, such that knowing $f$ in $A$ 
and $\widehat f$ in $B$ is sufficient to determine the function $f$ uniquely among all functions in a particular class. This leads to the following definition:
\begin{definition}
Given a function space $X\subset C^0(\mathbb R)\cap L^1(\R)$ and a pair of sets $A,B\subset \mathbb R$, we say that $(A,B)$ is a Fourier uniqueness pair
for $X$ if for any $f\in X$,
\begin{equation*}
f_{|A}=\widehat f_{|B}=0\Rightarrow f\equiv 0.
\end{equation*}
\end{definition} 
Even if we knew how to solve the problem above - that is, if we knew how to characterise \emph{exactly} the sets $A,B$ as before that form a Fourier uniqueness pair for a given function space $X$ -, that alone would not be enough to consider the problem solved in a satisfying manner for applications. Indeed, whenever we seek to reconstruct a function from time-frequency measurements based on a pair of sets $(A,B)$, we need a form of algorithm to do so. In other words, in addition to knowing that $(A,B)$ forms a Fourier uniqueness pair for a function space $X$, we shall need an \emph{interpolation formula} of the kind
\begin{equation}\label{eq : interpolation}
f(x)=\sum_{a\in A} f(a) h_a(x)+\sum_{b\in B} \widehat f(b)g_b(x).
\end{equation}
When this happens, we say that $(A,B)$ is an $X-$interpolation pair; if the right-hand side converges in the topology of $X$, we say that it is a strong $X-
$interpolation pair.

As effective and important examples of different areas where the problem above enters in a non-trivial way, we may cite, in addition to the ones stated in the beginning, the following:

\vspace{2mm}

\noindent \textbf{I.} \textit{Crystalline measures:} following Meyer \cite{Meyer}, a \emph{crystalline measure} is a Radon measure $\mu$ on $\mathbb R$
with discrete support such that the same is true for its Fourier transform $\widehat \mu$. In addition to the inherent mathematical interest in such structures stemming from generalising the Poisson summation formula, crystalline measures also arise naturally in physics \cite{Lagarias}. 

In order to see how interpolation formulas link with crystalline measures, note that an interpolation formula such as
\begin{equation*}
f(x)=\sum_{a\in A} f(a) h_a(x)+\sum_{b\in B} \widehat f(b)g_b(x)
\end{equation*}
may be written as
\begin{equation*}
\delta_x-\sum_A h_a(x)\delta_a=\mathcal F\left(\sum g_b(x) \delta_b\right).
\end{equation*}
This emphasises the fact that, in order to understand crystalline structures, a deep understanding of interpolation formulas is needed - and vice-versa (see \autoref{sec : remarks}).

\vspace{2mm}

\noindent \textbf{II.} \textit{Heisenberg uniqueness pairs:} A \emph{Heisenberg uniqueness pair} is a pair $(\Gamma, \Lambda)$, where
$\Gamma$ is an algebraic curve in $\mathbb R^2$ and $\Lambda\subset \mathbb R^2$ a set, such that, if we have a Radon measure $\nu$ supported on $\Gamma$  absolutely continuous with respect to the induced Lebesgue measure on $\Gamma$ 
and satisfying $\widehat \nu_{|\Lambda}=0$, then $\nu=0$. 
This concept, introduced by Hendemalm and Montes-Rodriguez in \cite{HMR}, is directly motivated by the study of \emph{unique continuation} of partial differential equations.

Methods from the Fourier uniqueness pairs theory have been immensely useful here: in \cite{GR}, Fourier uniqueness pairs where used to construct Heisenberg uniqueness pairs with $\Gamma$ a parabola and in \cite{bakan20} A. Bakan \textit{et al.} proved, for certain special classes of functions,  that a set $\Lambda$ forming a Heisenberg uniqueness pair with the hyperbola $\Gamma$ is the same as showing that a certain pair of set arising from $\Lambda$ form a Fourier uniqueness pair in some even dimension. This idea has also been exploited recently in \cite{RamosSto, RSto, bakan23} and \cite{radchenko2024}.

\vspace{2mm}

\noindent \textbf{III.} \textit{Universal optimality:} The recently solved problem of universal optimality of the two lattices $E_8$ and the Leech lattice (in $\mathbb R^8$ and $\mathbb 
R^{24}$, respectively) can be reduced,  
through linear programming, to constructing special \emph{interpolation formulas} for radial functions in the respective dimensions, with nodes appropriately placed at the lengths of vectors in the lattices $E_8$ and $\Lambda_{24}$ \cite{cohn2022}. 
In particular, in the \emph{sphere packing problem}, solved in 2016 in dimension 8 and 24 \cite{viazovska2017sphere, cohn2017sphere}, the aforementioned interpolation formulas may be used to explicitly resolve the question of existence of particular functions optimising the Cohn-Elkies linear programming bounds \cite{cohn2003new}. 

\vspace{2mm}

A first, classical, example of interpolation pair is given by the celebrated result due to Shannon, Whittaker,
Nyquist and others:
\begin{theoremA}[\cite{Higgins}]
Let $f\in L^2$ be such that $\widehat f(x)=0$ for a.e. $x\not \in [-1/2,1/2]$. Then $f$ is an entire function; moreover, we have
\begin{equation*}
f(x)=\sum_{n \in \Z} f(n)\frac{\sin(\pi (x-n))}{\pi(x-n)}.
\end{equation*}
\end{theoremA}
In other words, $(\mathbb Z,\mathbb R\setminus [-\frac12,\frac12]
)$ is a uniqueness pair for $\mathcal S$ and an interpolation pair for a suitable Paley-Wiener space. Many similar results followed (see the introduction of \cite{KNS} for more
examples), none of which having both $A$ and $B$ discrete. It
was only recently that Radchenko and Viazovska managed to 
obtain such a pair in \cite{RV}
(see \cite{Venkatesh} for a recent cohomological proof and \cite{Cohn} for a historical account and an overview of these results), where the authors actually constructed - using the theory of modular forms, inspired by \cite{viazovska2017sphere,cohn2017sphere} - a $\mathcal S_{even}-$interpolation formula
supported on $(\sqrt \N, \sqrt \N$), i.e. they proved that for any even Schwartz function $f$ we have
\begin{equation}
f(x)=\sum_{n\ge 0} f(\sqrt n)a_n(x)+\sum_{n\ge 1} \widehat f(\sqrt n)\widehat a_n(x)\label{EQ : RV}.
\end{equation}
However, their proof relies on symmetries that are unique to the set $\sqrt{\mathbb N}$ and, as such, it breaks for
different uniqueness pairs, for which new methods are needed \footnote{This is not to say that
this technique is not susceptible to generalisation: Indeed, Bondarenko, Radchenko and Seip did so in \cite{BRS}}. These new, purely analytical methods were first introduced by the second author and M. Sousa in \cite{RS1}, 
where they investigated the range of values $\alpha,\beta$ such that 
 $(\{\pm n^\alpha\}, \{\pm n^{\beta}\})$ is a Fourier uniqueness pair. Kulikov, Nazarov and Sodin \cite{KNS} later extended their
results to the whole  \emph{supercritical} range of exponents, providing also counterexamples in the subcritical range. It is interesting to note that one of the boundary cases of the main result in \cite{KNS} is \emph{exactly} the Radchenko\hspace{1pt}-\hspace{-1pt}Viazovska pair. 

In fact, in \cite{KNS} the authors proved much more: given any two non-decreasing sequences $\{x_n\}_n, \{y_n\}_n$ wich are $(p,q)$ supercritical, i.e. such that
 \begin{align*}
&\lim_{|n|\to \infty}|x_n|,|y_n|=\infty\\ &\limsup |x_n|^{p-1}|x_{n+1}-x_n|, |y_n|^{q-1}|y_{n+1}-y_n|<\frac12,\stepcounter{equation}\tag{\theequation}\label{eq: separated}
\end{align*}
 for some exponents $p,q > 1$ with $1/p + 1/q=1,$ then the pair $(A,B)$ with $A=\{x_n\}, B=\{y_n\}$ is a Fourier uniqueness pair for $\mathcal S$ and it even is a strong $\mathcal H-$Fourier interpolation pair, where $\mathcal H$ is 
 a particular 
 Sobolev space, so we have
 \begin{equation}\label{eq : interpolation-kns-weak}
f(x)=\sum_{\lambda\in A,\mu\in B} \left( f(\lambda)a_\lambda(x)+ \widehat f(\mu)b_\mu(x)\right).
\end{equation}
 On the other hand, they prove that if the pair of sequences is subcritical, i.e.
 \begin{equation*}
\liminf |x_n|^{p-1}|x_{n+1}-x_n|, |y_n|^{q-1}|y_{n+1}-y_n|>\frac12,
\end{equation*} then there exists a nonzero Schwartz function $f$ such that $f(x_n)=\widehat f(y_n)=0$, so $(A,B)$ is not a Fourier uniqueness pair.\par
The results of Kulikov, Nazarov and Sodin do not cover the critical case, i.e. the one where equality holds in \eqref{eq: separated}, in which
the Radchenko-Viazovska formula falls. Progress on the interesting problem of describing all the possible interpolation formulas (and Fourier uniqueness pairs) that satisfy this criticality condition was first made by the second author and M. Sousa, who in \cite{RS2} investigated perturbations of the interpolation formula due to Radchenko and
Viazovska, proving the following result, which settled a question raised by H. Cohn and D. Radchenko:
\begin{theoremA}[\cite{RS2}]\label{thm : per}
Given $\{\varepsilon_n\}_{n=0}^\infty$ with $\varepsilon_0=0, |\varepsilon_n| \le C n^{-\frac 54}\log^{-3}(1+n)$, 
$A=B=\{\sqrt{n+\varepsilon_n}\}$ is a Fourier uniqueness pair provided $C$ is small enough. Moreover, there exist functions 
$\theta_n, \nu_n$ such that for all $n\in\mathbb N$ we have
\begin{equation*}
|\theta_n(x)|+|\nu_n(x)|+|\widehat{\theta_n}(x)|+|\widehat \nu_n(x)|\le (1+j)^{\mathcal O(1)}(1+|x|)^{-10}
\end{equation*}
and for any $f\in\mathcal S_\text{even}$ we have
\begin{equation*}
f(x)=\sum f(\sqrt{n+\varepsilon})\theta_n(x)+\sum \widehat f(\sqrt{n+\varepsilon_n})\nu_n(x).
\end{equation*}
\end{theoremA}

\subsection{Main results} In this manuscript, we investigate the regularity required for the validity of the Fourier interpolation formulas. Specifically, if an interpolation formula such as \eqref{eq : interpolation}
holds, when is it a strong Fourier interpolation pair? What can be said
about the regularity of the interpolating functions?

 A case of particular interest is that of Theorem \ref{thm : per}:
as it turns out, we are able to improve the regularity of the basis functions
and the strength of convergence in that result all the way up to the Schwartz class in the \emph{subcritical} case of that theorem: 
\begin{theorem}\label{thm : pertSch}
Let $\{\varepsilon_n\}_{n\in\N}$ be a sequence such that $|\varepsilon_n|< c(1+n)^{-\frac54-\delta}$, for some $\delta>0$.
Then, if $c>0$ is small enough, there exists a family of Schwartz functions $\{h_n\}\cup \{g_n\}$ such that for any $f\in \mathcal S_{\text{even}}$
we have
\begin{equation*}
f(x)=\sum f(\sqrt{n+\varepsilon_n})h_n(x)+\widehat f(\sqrt{n+\varepsilon_n})g_n(x).
\end{equation*}
Moreover, the series converges absolutely in the Schwartz topology and such a decomposition is unique. 
\end{theorem}

This result improves drastically the ones in \cite{RS2}. Indeed, in \cite{RS2}, the authors only manage to prove that the perturbed interpolation functions belong to certain fixed-order Sobolev spaces, whereas Theorem \ref{thm : pertSch} shows that, as long as one has slightly more decay than in Theorem \ref{thm : per}, then the interpolating functions are automatically as regular as one wishes. 

This confirms the regular behaviour already displayed by the original basis functions \cite{RV}, which is perhaps surprising at first glance, due to the seemingly irregular nature of the perturbations at hand. 

 The same question can be asked of the supercritical case. In that regard, we note that in \cite{KNS}, in spite of constructing interpolation formulas such as \eqref{eq : interpolation-kns-weak}, the authors remark that their method
does \emph{not} imply the existence of interpolation formulas
converging in $\mathcal S$ in the supercritical case, and they hence leave as an open question the task of transforming their interpolation results into truly \emph{supercritical Schwartz-converging} formulas. 

In the next result, we are able to answer their question in the \emph{affirmative},  at least for sufficiently dense sequences - i.e.  in a certain subset of the supercritical range. 

\begin{theorem}\label{thm : KNSsoleven} Let $\{x_n\}_{n \in \Z}$ and $\{y_n\}_{n \in \Z}$ denote two even sequences of real numbers satisfying 
\begin{equation}\label{eq : bound-x-y-stronger} 
\max\left( \limsup_{n \to \infty} |x_{n+1}|^D |x_{n+1} - x_n|, \limsup_{n \to \infty} |y_{n+1}|^D |y_{n+1} - y_n| \right) \le c. 
\end{equation}
Then, if $D > \frac{7}{2}$ and $c>0$ is sufficiently small, there exist functions $h_n,g_n \in \mathcal S$ such that for all functions in $f\in\mathcal S$,
\begin{equation*}
f(x)=\sum f(x_n) h_n(x)+\sum \widehat f(y_n)g_n(x),
\end{equation*}
and the sum converges in $\mathcal S$.
\end{theorem}

We note that this result and its proof provide an \emph{automatic} improvement argument, showing how better estimates on the Radchenko-Viazovska basis functions imply a positive answer in an even larger number of cases. That is, if one improves the dependency on $n$ in the upper bound on the functions $\{a_n\}_{n \in \N}$ in Theorem \ref{thm : RV} (see, for instance, Theorem \ref{thm : RS2est} below), then Theorem \ref{thm : KNSsoleven} is \emph{automatically} improved. 

One of our main tools to prove Theorem \ref{thm : KNSsoleven} - as well as Theorem \ref{thm : pertSch} - is the framework of \emph{sc-Fredholm operators}. These are operators defined on a `scale' of Hilbert spaces, which are, roughly speaking, Fredholm at each level, and \emph{regularising} - in the sense that, if the operator takes an element from a (possibly) lower Hilbert space into a higher one, then the original element has to lie `higher' in the scale, and hence have more `regularity'. 

As it turns out, the scales of Hilbert spaces both in \cite{RS2} and in \cite{KNS} lie \emph{exactly} in the realm of regularising operators in our case. That is, if we define, in loose terms, the operator $T$ that maps a Schwartz function (resp. a function belongin to a certain high order Sobolev space) into a sequence $(f(x_k),\widehat{f}(y_k))_k$, for a fixed pair of sequences $(\{x_k\}_k,\{y_k\}_k)$ satisfying a suitable \emph{supercritical condition} \`a la \cite{KNS}, then we may show by the main results in that manuscript that $T$ is regularising. 

Moreover, if one is given a bit more of structure - such as that provided by the perturbative framework of the proofs in \cite{RS2} -, then one is able to show that the operators at hand are also \emph{Fredholm} at each level. Using these two crucial points, we then conclude that the subspace generated by the action of the evaluation operators considered is, on the one hand as a consequence of the results from \cite{KNS}, \emph{closed}, and, on the other hand by the fact that the operator itself is Fredholm in the perturbative regime, that that subspace $T(\mathcal{S})$ is \emph{complemented} within a suitable space of sequences. 

As a last ingredient in the proof, we are then naturally led to invoke the theory of \emph{Schauder frames}. This theory applied to our case predicts, in a nutshell, that the convergence of the various interpolation formulas of Theorem \ref{thm : pertSch}, \ref{thm : KNSsoleven} in the Schwartz space is equivalent to the existence of a \emph{strong Schauder frame} which itself turns out to be \emph{equivalent} to the complementedness of the subspaces generated by our operator, which falls exactly within the particular collection of sequences which we are able to simultaneously treat with tools stemming from \cite{KNS} and \cite{RS2}. 


\vspace{2mm}

Finally, as a last contribution of this manuscript, we deal with the case of even \emph{more regular} interpolation formulas. Indeed, one may prove that both $a_n$ and $\widehat a_n$ may be extended as \emph{entire and of order $2$} on $\C$, so it is a natural question whether there exists a decay rate fast enough of perturbations which implies the analyticity of the respective perturbed interpolation functions. 

Contrarily to the previous results in this manuscript, it seems that such a result \emph{cannot} be proved with purely functional-analytic techniques, such as the other ones used to prove Theorems \ref{thm : pertSch} and \ref{thm : KNSsoleven}, as there is no natural space of analytic functions in which the $a_n$ functions are a basis - or even an ``almost-basis'', for that matter. This is due to the fact that the decay results obtained in \cite{RS2} (see Theorem \ref{thm : RS2est}) are, as a matter of fact, \emph{best possible} - see Proposition \ref{prop : bound-decay-b-n} below for a proof of this fact.  

It is then most surprising that by assuming nothing more than exponential decay of perturbations, we are able to retain this property of the original Radchenko--Viazovska basis functions. This is the content of the next result: 

\begin{theorem}\label{thm : entire}
Given $\{\varepsilon_n\}_n$ a sequence such that  $|\varepsilon_n| \leq ce^{-Cn},$ with $c>0$ sufficiently small and $C>0$ sufficiently large, we have that the perturbed interpolating functions $\{h_n\} \cup\{g_n\}$ from Theorem \ref{thm : pertSch} are \emph{entire of order $2$}. Moreover, they can be seen to be of \emph{finite type equal to $\pi,$} where we define the type of an entire function $G$ of order $2$ to be 
\[
\sigma = \limsup_{r \to \infty} \frac{\log(\sup_{z \in B_r(0)} |G(z)|)}{r^2}.
\]
\end{theorem} 

This result shows that the analyticity property of the interpolation functions is, in a sense, \emph{robust}. In particular, it implies the surprising fact that, if $f \in \mathcal{S}_{even}$ satisfies 
\[
f(\sqrt{n+\varepsilon_n}) = \widehat{f}(\sqrt{n+\varepsilon_n}) = 0, \, \forall n\in\N \text{ such that } |n| \ge N_0,
\]
then $f$ is automatically analytic of order 2 and type $2\pi$. This may be seen as an \emph{extension} of the main principles in the first part of the proof of the main results of \cite{RS1,KNS}, which states that sufficiently dense zeros in both space and frequency in a \emph{supercritical} sense imply analyticity. 

In order to prove Theorem \ref{thm : entire}, we need different methods from the ones in the other main results of the manuscript. Indeed, Theorem \ref{thm : entire} is obtained by an application of the framework of \emph{modular form estimates} for the basis functions, as carried out in \cite{RS2,RV}.  By employing the exact formulas for the generating functions of the modular forms giving rise to the interpolation functions $\{a_n\}_{n \in \N}$, we are able to show that these functions are not only entire of order 2, but that the growth of their ``norm'' as analytic functions, in a suitable sense, is bounded by an exponential in $n$. 

Leveraging this fact, together with another application of perturbative techniques from functional analysis, allows us to show that any sufficiently small exponential perturbation of the basis functions remains analytic of the same order, completing the proof of Theorem \ref{thm : entire}. 

\vspace{2mm}

This manuscript is organized as follows. In Section \ref{sec : prelim}, we discuss preliminary facts which shall be important for the proof of our main results. We divide the proof of the main results in three parts: Theorem \ref{thm : pertSch} is dealt with in Section \ref{sec : schw-g-s}, Theorem \ref{thm : KNSsoleven} is proved in Section \ref{sec : schw-int}, while Theorem \ref{thm : entire} is discussed in Section \ref{sec : analytic}. Finally, we comment on possible extensions of our methods to other cases, as well as obstructions to generalising the current methods to other settings, in Section \ref{sec : remarks}. 

\vspace{2mm}

\subsection*{Acknowledgements} We would like to thank A. Figalli for several comments, questions and remarks that helped us improve the exposition of this manuscript. We would also like to thank J. Bonet for discussions related to the results on Schauder frames needed for this work, as well as D. Radchenko, for discussing the optimal decay rates of interpolation functions in \cite{RV}, which gave rise to Proposition \ref{prop : bound-decay-b-n}. 

\section{Preliminaries}\label{sec : prelim}


\subsection{Schauder frames and bases} We start our discussion on preliminaries with some background material on Schauder bases and Schauder frames. In addition to the many references already contained in the text below, we also point the reader to other further recent developments in that theory, such as \cite{freeman2021schauder, eisner2020continuous, beanland2012upper}.

\begin{definition}[Bases]
Let $X$ be a locally convex space, $\{e_n\}\subset X$ a linearly independent subset of $X$ and $c_n$ a sequence of linear (but not necessarily continuous) functionals
 $c_n:X\to \mathbb K$, where $\mathbb{K}$ denotes either $\R$ or $\C$. We say $X$ is a topological basis if for every $x\in X$ 
\begin{equation}
x=\sum c_n(x) e_n, \label{eq : sd}
\end{equation}
where the right hand side converges in the topology of $X$ and the decomposition is unique.\\

We say that $\{e_n\}$ is a \emph{Schauder basis} if each of the linear functionals $c_n$ is continuous. A Schauder basis is said to be \emph{absolute} if
the sum in \eqref{eq : sd} converges absolutely, and it is \emph{strong} if for any bounded set $B$ and any continuous seminorm $p$ we have
\begin{equation*}
\sum p_B(c_n)p(x_n)<\infty,
\end{equation*}
where $p_B(f):=\sup_{x\in \text{co}(B)}|f(x)|$ .
\end{definition}

Directly related to this concept is that of a \emph{Schauder frame}: as we will see, not all the interpolation formulas induced by Fourier uniqueness pair are associated to bases. The reason for this is that we may incur in ``oversampling'', in 
some sense. This leads to the necessity of developing a slightly more flexible notion, that of frame (see \cite{Grochenig}, chapter 5 for background on the theory of frames for 
Hilbert spaces), for locally convex spaces; this theory is fairly recent, 
having been introduced in the papers \cite{Bonet1, Bonet2, FreFra}. We refer the reader also to the references therein. 

In the remainder of this section, we mostly follow \cite{Bonet2}.

\begin{definition}[Schauder frame]
Let $\{x_n\}$ be a countable collection of vectors in a locally convex space $X$ and $\{y_n\}$ a countable collection of linear functionals in $X'$. We say $\left((x_n,y_n)\right)_n$ is a Schauder frame if, for all $x\in X$ we have
\begin{equation*}
\sum y_n(x)x_n=x,
\end{equation*}
where the left hand side converges in the topology of $X$.
\end{definition}

It is interesting to note that Schauder frames provide a full characterisation of spaces with the approximation property: it is well known that
spaces admitting a Schauder basis have the bounded approximation property and that the opposite is not true; however, having the
bounded approximation property is equivalent to admitting a Schauder frame (see \cite{Bonet1}, corollary 1.5). This in particular implies
the existence of nuclear Fréchet spaces that admit a Schauder frame but not a Schauder basis.

We next explore the related definition of frames with respect to a fixed sequence space. 

\begin{definition}[$\Lambda$ frame]
Let $X$ be a locally convex space and let $\{y_n\}$ be a countable collection of elements of $X'$. Given a sequence space $\Lambda$, we say that $(y_n)$ is a $\Lambda-$frame if the map $T$ defined as $x\mapsto (y_n(x))_n$ is an embedding of $X$ in $\Lambda$. We say $(y_n)$ is a strong $\Lambda-$frame for $X$ if $T(X)$ is complemented in $\Lambda$.
\end{definition}
\begin{proposition}\label{prop : strfr}
Assume $\Lambda$ is a Fréchet
sequence space such that the canonical sequence $(e_n)$ is a Schauder basis and $\{y_n\}_n$ is a $\Lambda-$frame for $X$ and assume $X$ to be complete. Then there exists a sequence $\{x_n\}_n$ such that
$(x_n,y_n)$ is a Schauder frame and the sequence $\sum c_n x_n$ converges for all $x\in \Lambda$ if and only if $\{y_n\}$ is a strong $\Lambda-$frame.
\end{proposition}
\begin{proof}
Assume $\{y_n\}$ is a strong $\Lambda-$frame; let $x_n:=T^{-1}(P(e_n))$, where $P$ is the projection from $\Lambda$ to $T(X)$ (which exists
since $T(X)$ is complemented). Then, ignoring for the moment the issues of convergence, we have formally
\begin{equation*}
T\left(\sum x_n y_n(x)\right)=\sum y_n(x) P e_n=P\left(\sum y_n(x)e_n\right)=P(T(x))=T(x).
\end{equation*}
To deal with the issues of convergence it is sufficient to observe that 
\begin{equation}
\sum_{0\le n\le k} x_n y_n(x)
=T^{-1}\circ P\left(\sum_{0\le n\le k} u_n(x)e_n)
\right)
\end{equation}
and since the series in the right hand side converges, so does the 
left hand side.\par
On the other hand, if there exists a sequence $(x_n)$ such that $(x_n,y_n)$ is a Schauder frame, then
the following $P$ provides the desired projection:
\begin{equation*}
P((y_n)):=T^{-1}\left(\sum y_n x_n\right)
\end{equation*}
and $P$ is bounded by the uniform 
boundedness principle.
\end{proof}
\begin{corollary}\label{cor : subframe}
Let $\Lambda$ be a Fréchet sequence space as in \ref{prop : strfr}, $X$ a complete locally convex space and $\{y_n\}$ a $\Lambda-$frame for $X$. If there exists a subsequence $y_{n_k}$
which induces a strong $\Lambda-$frame, then $\{y_n\}$ is a strong $\Lambda-$frame.
\end{corollary}
\begin{proof}
Let $x_{n_k}$ be the elements of $X$ whose existence is implied by proposition \ref{prop : strfr}. Then define
\begin{equation*}
x_n:=\begin{cases}x_{n_k}&\text{ if }n=n_k\\
0&\text{ otherwise}
\end{cases}
\end{equation*}
It is easy to see that $(x_n, y_n)$ is a Schauder frame and $\sum c_n x_n$ converges, hence $\{y_n\}$ is a strong $\Lambda-$frame.
\end{proof}
\begin{remark}
The results above hold, more 
generally, if $\Lambda$ is a 
barrelled locally convex sequence space with 
$\{e_n\}$ a Schauder basis.
It is natural to wonder, given the results on bases, if a Schauder frame in a nuclear Fréchet space must be an absolute Schauder frame. As it turns out, this problem is currently open. 

To give the reader an idea of the complexity of the problem, let us mention a slightly more general version that is known to 
be false: a sequence $\{x_n\}\subset X$ is said to be a \emph{representing system} in $X$ if for each $x\in X$ there exists a sequence $(\alpha_n)$ such that
\begin{equation*}
x=\sum \alpha_n x_n.
\end{equation*}
It is easy to see that the concept is a generalisation of both Schauder bases and frames; in this setting, it is not true that every representing system
in a Fréchet nuclear space is an absolute system, as shown by Kadets and Korobeinik in \cite{Kadets}.

\end{remark}

\vspace{2mm}

\subsection{Fredholm Operators and sc-Fredholm Operators} We next lay out the foundation for the application of the theory of sc-Fredholm operators, as mentioned in the introduction. We first recall the definition of semi-Fredholm and Fredholm operators:
\begin{definition}
    Let $X,Y$ be locally convex spaces. An operator $T:X\to Y$ is said to be upper (respectively, lower) semi-Fredholm
    if its range is closed and its kernel finite-dimensional (respectively, its kernel is closed and its range has finite codimension). 
    An operator is Fredholm if it is both upper and lower semi-Fredholm. The class of upper
    (respectively lower) semi-fredholm operators
    is denoted $\Phi^{+}(X,Y)$ (resp.$\Phi^{-}(X,Y)$). The class of Fredholm operators will be denoted $\Phi(X,Y).$
\end{definition}
In what follows we will be interested in the following
problem: we have two scales of Hilbert spaces $(X_i),
(Y_i)$ (i.e. nested sequences of Hilbert spaces with 
compact inclusions) and an operator$ T$ acting on $T:X_i\to Y_i$ for all $i$. If $T$ is Fredholm for all
$i$ with constant index, when is it true that $T$
is Fredholm when acting on the projective limits? 
A sufficient condition for this to happen
is given by the concept of scale Fredholm operator:
\begin{definition}[\cite{Hofer2007}, section 2.1]
Let $\mathbb X=(X_i), \mathbb Y=(Y_i)$ be scales of Hilbert spaces. An operator $T$ is said to be an $sc-$operator between $\mathbb X$ and $\mathbb Y$
if $T\in \mathcal L(X_i,Y_i)$.
A subspace $C$ of $X_0$ is said to be an $sc-$subspace
if $(C\cap X_i)$ is a scale of Hilbert spaces. Given
two such subspaces $C,D$ we say that $\mathbb X=C\oplus_{sc}D$ if at each level $i$ of the scale
we have $X_i=C_i\oplus D_i$.
An operator $T:X_i\to Y_i$ is said to be sc-Fredholm
if there exist splittings $\mathbb X=K\oplus_{sc} \text{ker}(T)$ and $\mathbb Y=K'\oplus_{sc}T(\mathbb X)$ with both $\text{ker}(T)$ and $K'$ finite dimensional and such that $T_{|K}:K\to T(\mathbb{X})$ is a 
sc-isomorphism.
\end{definition}
It is clear that an $sc-$Frehdolm operator is also
Fredholm at the projective limit.
\begin{proposition}[\cite{wehrheim2016}, lemma 3.6]\label{thm : Fredholm}
   Let $\mathbb X,\mathbb Y$ be two scales of Hilbert
   spaces and $T$ an $sc-$operator between them.
   $T$ is sc-Frehdolm if and only if $T\in 
   \Phi(X_0,Y_0)$ and $T$ is regularising, i.e. $x\in 
   X_0$ together with $T(x)\in Y_i$ implies $x\in X_i$.
\end{proposition}
\begin{remark}
The operators we will consider
in this work are always regularising, as one can see thanks to \cite[Theorem~2]{KNS} and the remarks following that result and also to \cite[Theorem 1.7 and 1.8]{RS2}.
\end{remark}

\begin{corollary}[\cite{Hofer2007}, Prop 2.11]
In the same setting as the previous
theorem, let $K:\mathbb X\to \mathbb Y$ be an sc-operator
such that $K(X_i)\subset Y_{i+1}$.
Such an operator is called $sc^+$
and perturbations under 
$sc^+$ operators preserve $sc-$Fredholm operators and their indexes.
\end{corollary}

\vspace{2mm}

\subsection{Neumann series and perturbation of frames}  We conclude this section with a generalised version of the Neumann series trick and an application of the generalised trick to perturbation
of frames.
\begin{theorem} Let $T:X\to X$ be an operator on a complete locally convex space and assume that there exists  a continuous seminorm $p_0$ such that 
for all continuous seminorms $p$ there exists a constant $C_p$ such that
\begin{equation*}
p(Tx)<C_pp_0(x).
\end{equation*}
Assume moreover that $C_{p_0}<1$. Then $I-T$ is invertible.
\end{theorem}
\begin{proof}
It suffices to prove that the Neumann sum $\sum T^n x$ converges absolutely. To do so, let $p$ be any continuous seminorm. Then
\begin{equation*}
\sum_{n=0}^\infty p(T^n(x))\le p(x)+ C_p\sum_{n=1}^\infty p_0(T^{n-1}(x))\le p(x)+C_p\sum_{n=1} C_{p_0}^{n-1}p_0(x)<\infty
\end{equation*}
and this proves the result.
\end{proof}
As a corollary, we obtain the following perturbation result for Schauder frames and bases:
\begin{corollary}\label{cor : pertframe}
Let $X$ be a complete locally convex space and let $(x_n,y_n)$ be a Schauder 
frame. Let $\tilde y_n$ be a family of continuous functionals on $X$ and 
assume that there exists a continuous seminorm $p_0$ such that for all 
continuous seminorms $p$
\begin{equation*}
\sum |(y_n-\tilde y_n)(x)|p(x_n)<C_pp_0(x)
\end{equation*}
and that $C_{p_0}<1$. Then there exist vectors $\{\tilde x_n\}$ such that 
$(\tilde x_n, \tilde y_n)$ is a Schauder frame. If $(x_n,y_n)$ is a basis,
so is $(\tilde x_n, \tilde y_n)$.
\end{corollary}
\begin{proof}
    Define the operator $T(x)=\sum (y_n-\tilde y_n)(x)x_n$. It follows by our
    hypotheses that $T$ is well defined and that $I-T$ is invertible.
    It then suffices to take  $\tilde x_n=(I-T)(x_n)$ and by bijectivity of $I-T$ it follows that $(\tilde x_n,\tilde y_n)$ is a Schauder frame. 
\end{proof}
 This result will allow us to extend the same
technique used in \cite{RS2} to analyse perturbations
in more general locally convex spaces; the drawback will be that
such perturbations will need to satisfy not just one
inequality induced by a certain norm estimate but rather
an infinite amount of them, induced by each seminorm at a different scale. We will then see that a first naive use of the
above corollary can only get us so far and to get stronger
results (such as the Schwarz convergence) a more refined analysis will be needed and that will be where more functional-analytical arguments, such as scale operators, will enter into play.

\subsection{Modular forms}\label{ssec : modular} In this subsection, we gather some of the facts we will need about modular forms, especially in Section \ref{sec : analytic}. For more details, we refer the reader to \cite{Chandrasekharan}and \cite[Section~2]{RV}; see also\cite{BN, Zagier}.

We define the action of the group $SL_2(\R)$ of matrices with real coefficients and determinant 1 on the upper half-plane through M\"obius transformations: for 
\[
\gamma = \begin{pmatrix}
         a & b \\ c & d 
         \end{pmatrix} \in SL_2(\R), \, z \in \C_+ \Rightarrow \gamma z = \frac{az + b}{cz+d} \in \C_+. 
\]
Some elements of this group will be of special interest to us. Namely, we let 
\[
I = \begin{pmatrix} 
    1 & 0 \\ 0 & 1 
    \end{pmatrix}, 
    \,\, T = \begin{pmatrix}
              1 & 1 \\ 0 & 1 
             \end{pmatrix}, \,\, S = \begin{pmatrix}
				  0 & -1 \\ 1 & 0
				  \end{pmatrix}
\]
This already allows us to define the most valuable subgroup of $SL_2(\Z)$ for us: the group $\Gamma_{\theta}$ is defined then as the subgroup of $SL_2(\Z)$ generated by 
$S$ and $T^2.$ We now define the so-called Jacobi theta series. These are defined in $\C_+$ by
\begin{align*} 
\Theta_2(z) &= \sum_{n \in \Z + \frac{1}{2}} q^{\frac{1}{2} n^2},\cr
\Theta_3(z) (= \theta(z)) &=\sum_{n \in \Z} q^{\frac{1}{2} n^2}, \cr 
\Theta_4(\tau) &= \sum_{n \in \Z} (-1)^n q^{\frac{1}{2}n^2}.\nonumber
\end{align*}
Here, $q = q(z) = e^{2\pi i z}.$ These functions satisfy the identity $\Theta_3^4 = \Theta_2^4 + \Theta_4^4.$ Moreover, under the action of the elements $S$ and $T$ of $SL_2(\Z),$ they transform as 
\begin{align}
\nonumber (-iz)^{-1/2} \Theta_2(-1/z) &= \Theta_4(z), \,\,\Theta_2(z+1) = \exp(i\pi/4) \Theta_2(z), \\
\label{eq theta transform} (-iz)^{-1/2} \Theta_3(-1/z) &= \Theta_3(z), \,\,\Theta_3(z+1) = \Theta_4(z), \\
\nonumber (-iz)^{-1/2} \Theta_4(-1/z) &= \Theta_2(z), \,\,\Theta_4(z+1) = \Theta_3(z). 
\end{align}
These functions allow us to construct the classical lambda modular invariant given by 
\[
\lambda(z) = \frac{\Theta_2(z)^4}{\Theta_3(z)^4}.
\]
The lambda invariant satisfies the following $q-$expansion at $i \infty$
\begin{equation}\label{eq lambda}
 \lambda(z) = 16q^{1/2} - 128q + 704q^{3/2} + O(q^2).
\end{equation}
{The function $\lambda$ is also invariant} under the action of elements of the subgroup $\Gamma(2) \subset SL_2(\Z)$ of all matrices $\begin{pmatrix}                         a& b \\ c & d 
\end{pmatrix}$ so that 
$a\equiv b \equiv 1 \mod 2, \, c \equiv d \equiv 0 \mod 2$, { and $\lambda(z)$ never assumes the values $0$ or $1$ for $z\in\C_+$}. Besides this invariance, \eqref{eq theta transform} gives us immediately that 
\begin{align}\label{eq lambda transform}
{\lambda(z+1) = \frac{\lambda(z)}{\lambda(z)-1}}, \, \, \lambda\left(-\frac{1}{z}\right) = 1-\lambda(z).
\end{align}
{We then define the following modular function for $\Gamma_{\theta}$:}
\[
J(z) = \frac{1}{16} \lambda(z) (1-\lambda(z)).
\]
From \eqref{eq lambda transform}, we obtain that $J$ is invariant under the action of $\Gamma_{\theta}:$  
\[
J(z+2) = J(z), \,\, J\left(-\frac{1}{z}\right) = J(z).
\] 
We note that the function $J$ has the following $q-$expansion at $i\infty$:
\begin{equation}\label{eq : J-dec} 
J(z) = q^{1/2} - 24q + O(q^{3/2}).
\end{equation}
Finally, another useful fact will be the $q-$expansion of the \emph{reciprocal} of $J$ at $i\infty$: 
\begin{equation}\label{eq : J-inv-dec} 
\frac{1}{J(1-1/z)} = -4096q + O(q^2).
\end{equation}

\subsection{Bounds on Fourier Interpolation Bases} First, recall the exact formulation of the main result in \cite{RV}: 
\begin{theorem}[Theorem~1 in \cite{RV}]\label{thm : RV} 
There exist a sequence of Schwartz functions $(a_n)_n$ such that $a_0=\widehat{a}_0$ and for all even $f\in\mathcal S$ we have
\begin{equation}
f(x)=\sum f(\sqrt n)a_n(x)+\sum \widehat f(\sqrt n)\widehat a_n(x), \label{EQ : RVint}
\end{equation}
where the right hand side converges pointwise absolutely. Moreover, the functions $a_n$ satisfy the following properties: given $n,m>0$ we have
\begin{align*}
a_{n}(\sqrt m)&=\delta_{n}^m, \quad\widehat a_{n}(\sqrt m)=0, \\
\quad a_0(0) & =\frac12,\widehat a_{n}(0)=\mathbbm 1_{\exists k\in\mathbb N\  n=k^2}\\
a_n(0)&=-\mathbbm 1_{\{n \text{ is a perfect square}\}},
\end{align*}
and both $a_n$ and $\widehat a_n$ are entire functions or order $2$.
\end{theorem}

In \cite{RV}, the authors proved that the functions $\{a_n\}_{n \in \N}$ belong the the Schwartz class, with a uniform bound of the kind $|a_n(x)| = O(n^2)$ for all $n \in \N$ and all $x \in \R$. By sharpenning their methods and employing other modular form tools, such as the ones from \cite{BRS}, the second author and M. Sousa were able \cite{RS2} to prove the following improved decay bound on the derivatives of functions in the interpolation basis:

\begin{theorem}[Corollary~4.6 in \cite{RS2}]\label{thm : RS2est}
There exist two constants $c, C$ such that for $n\ge 1$ we have
\begin{equation}
|a_n'(x)|+|\widehat a_n'(x)|\lesssim n^{\frac 34}\log^3(1+n) \left(e^{-c\frac{|x|^2}n}\mathbbm 1_{|x|\le Cn}+e^{-c|x|}\mathbbm 1_{|x|>Cn}\right)\label{eq : RS2est}.
\end{equation}
Moreover, we have that
\begin{equation*}
|a_0(x)|\lesssim e^{-\sqrt{\frac\pi2}|x|}, \, \forall \, x \in \R.
\end{equation*}
\end{theorem}

Finally, we note that the previous result directly implies a bound on the \emph{Schwartz seminorms} of the interpolation functions, which will be useful in the proofs of the main results below. 

\begin{corollary}\label{cor : seminorms}
Given $\alpha, \beta\in\mathbb N$ we have
\begin{equation*}
\|x^\alpha \partial_\beta a_n\|_\infty+\|x^\alpha \partial_\beta \widehat a_n\|_\infty\lesssim_{\alpha,\beta} n^{\frac{\alpha+\beta}2}n^{\frac34}\log(1+n)^3.
\end{equation*}
\end{corollary}
\begin{proof}
We define a weight function $\omega:[0,\infty)\to [0,\infty)$ to be a function which satisfies 
\begin{enumerate}
\item $\omega(2x)=O(\omega(x))$ for $x\to \infty$
\item $\log(x)=o(\omega(x))$ for $x\to \infty$
\item $\varphi(x):=\omega(e^x)$ is convex on $[0,\infty)$
\end{enumerate}
We then define the \emph{Gelfand-Shilov norm} of a function $f$ at height $h$ with respect to $\omega$ as 
\begin{equation*}
\|f\|_{S^h_{\omega}} :=\|f(x) e^{h\omega(|x|)}\|_{\infty}+\|\widehat f(\xi)e^{h\omega(|\xi|)}\|_\infty.
\end{equation*}
In case $\omega(x) = |x|^{\alpha},$ we shall write $\|f\|_{S^h_{|\cdot|^{\alpha}}} =:\|f\|_{S^1_{1,h}}.$ It is well-known \cite{Chung96} that the characterization above of the Gelfand-Shilov norm implies (and is in fact equivalent to) 
\[
\|x^{\alpha} \partial^{\beta}f\|_{\infty} \lesssim (C \cdot h)^{|\alpha| + |\beta|} \cdot \alpha^{\alpha} \cdot \beta^{\beta} \|f\|_{S^1_{1,h}}. 
\]
Using now \eqref{eq : RS2est} it is easy to see that, for $n>c^{-2}$
\begin{align*}
\|a_n\|_{S^1_{1,\frac1{\sqrt n}}}&\le \|(e^{-c\frac{|x|^2}n}+e^{-c|x|})e^{\frac{|x|}{\sqrt n}}\|_\infty n^{\frac34}\log^3(1+n)\lesssim n^{\frac34}\log^3(1+n).
\end{align*}
Thus, by using the description of the $S^1_{1,h}$-norm above, we get that
\begin{equation*}
\sup \frac{\|x^\alpha \partial_\beta a_n\|_\infty}{\left(\frac{c}{\sqrt n}\right)^{\alpha+\beta}\alpha^\alpha\beta^\beta}\lesssim n^{\frac34}\log^3(1+n)
\end{equation*}
and similarly for $\widehat{a}_n$, finishing our proof. 
\end{proof}
These results imply, in particular, that the Radchenko-Viazovska functions form a Schauder basis of $\mathcal S$:

\begin{corollary}
    Let $x_0=\frac{a_0+\widehat a_0}2$, $x_{n>0}=a_n, y_{n>0}=\widehat a_n$. Then the set $\{x_n\}_{n\in\mathbb N}\cup \{y_n\}_{n=1}^\infty$ is a Schauder basis of $\mathcal S_{\text{even}}$,
\end{corollary}

\subsection{Notation and conventions}\label{ssec : weights} Whenever we write $a \lesssim b$, we mean that $a \le C \cdot b$, with $C$ an absolute constant; if instead we write $a \lesssim_{\alpha,\beta,\gamma} b$, it is meant to be understood as $a \le b$ with $C$ being a constant depending on the parameters $\alpha,\beta,\gamma$. 

Throughout this manuscript, we shall work with several different function spaces, the most prominent of which being the sequence spaces 
\[
\ell^2_s(\N) := \left\{ \bx \in \C^{\N} \colon \sum_{n \ge 1} |\bx_n|^2 (1+n)^{2s} < +\infty \right\},
\]
with the natural induced norm. We then define the space 
\[
\mathfrak{s} := \left\{ \bx \in \C^{\N} \colon \lim_{n\to \infty} |\bx_n|\cdot n^k = 0, \, \forall \, k > 0\right\}.
\]
These spaces of sequences are natural counterparts to \emph{Lebesgue-Sobolev spaces} and the Schwartz class: indeed, we define for each $t > 0$
\[
H^t(\R) : = \left\{ f \in \mathcal{S}'(\R) \colon \int_{\R} |\widehat{f}(\xi)|^2 (1+|\xi|)^{2t} \, \diff \xi < +\infty\right\}.
\]
With this definition, we define for $p, q > 1, p^{-1} + q^{-1} = 1,$ the spaces
\[
\mathcal{H}_{p, q}^s:=\left\{f\in \mathcal{S}'(\R)\colon f \in H^{q s}, \widehat{f} \in H^{p s},\|f\|_{\mathcal{H}_{p, q}^s}^2=\|f\|_{H^{p s}}^2+\|\widehat{f}\|_{H^{q s}}^2\right\}.
\]
In the particular case of $p=q=2,$ we shall denote ${\bf H}^s := \mathcal{H}_{\frac{s}{2}, 2,2}.$

In the same way as the Schwartz space is the projective limit of ${\bf H}^s$ with $s \to \infty$, we have that 
$$\mathfrak{s} = \lim_{\infty \leftarrow s} \ell^2_s(\N).$$
This last fact will be of pivotal importance in our proofs below.

Finally, let us also note that the crucial fact that the embeddings 
$$\ell^2_s(\N)\hookrightarrow \ell^2_{s'}(\N), \, s>s'\ge 0$$ are compact: this follows from the fact that $\ell^2_s(\N)$ can be identified with $H^s(\T)$ naturally through Fourier series, together with the Rellich-Kondrachov embedding theorem. 

\section{Schwartz regularity of perturbed basis functions}\label{sec : schw-g-s}

\subsection{Proof of Theorem \ref{thm : pertSch}} In order to prove Theorem \ref{thm : pertSch}, we will maintain the idea of perturbing our basis via an operator whose invertibility we prove via Neumann series, as done in \cite{RS2}, but this will be accomplished by taking into account the spaces $\mathcal H^s$, $\ell^2_s$ in the framework of $sc-$Fredholm operators.

\begin{proof}[Proof of Theorem \ref{thm : pertSch}]
{\bf Step 1.} We first prove the result with the additional hypothesis that $\varepsilon_0=0$. \\

Let $\mathcal R:\mathcal H^s\to \mathbb C\oplus\ell^2_s\oplus \ell^2_s$ be the formal analysis operator associated to our basis; namely
\begin{equation*}
\mathcal R(f):=((f(0)+\widehat f(0)), (f(\sqrt n))_n, (\widehat f(\sqrt n))_n).
\end{equation*}
While this operator is only densely defined (a priori, it is only defined on $\mathcal S_{even}$ where it induces an isomorphism with $\mathbb C\oplus \mathfrak s\oplus \mathfrak s$), it suggests the following approach:
 we will construct an isomorphism $T$ of $\C\oplus \ell_s^2\oplus \ell_s^2$ which satisfies
\begin{equation*}
T\circ \mathcal R(f)=((f(0)+\widehat f(0)),(f(\sqrt{n+\varepsilon_n}))_n,(\widehat f(\sqrt{n+\varepsilon}))_n).
\end{equation*}
We will then prove that $T$ is continuous on $\mathbb C\oplus \mathfrak s\oplus \mathfrak s$ and invertible there as well, thus proving the theorem since we 
can obtain the basis by inverting (and since basis expansions in
Fréchet nuclear spaces converge absolutely).
To construct $T$, let 
$$F(z,\bx,\by):=za_0(x)+\sum_{i=1}^\infty \bx_{i} a_i(x)+\sum_{i=1}^\infty \by_{i} \widehat a_i(x),$$
and define $T$ as $T(z, \bx,\by)=(T_0,T_1,T_2)$ with
\begin{align*}
T^0(z,\bx,\by)&:=z=F_{z,\bx,\by}(0)+\widehat F_{z,\bx,\by}(0)\\
T^1(z,\bx,\by)&:=(F_{z,\bx,\by}(\sqrt{1+\varepsilon_1}), \dots, F_{z,\bx,\by}(\sqrt{n+\varepsilon_n})\dots)\\
T^2(z,\bx,\by)&:=(\widehat F_{z,\bx,\by}(\sqrt{1+\varepsilon_1}), \dots, \widehat F_{z,\bx,\by}(\sqrt{n+\varepsilon_n})\dots)=T^1(z,\by,\bx).
\end{align*}
We now calculate the Hilbert-Schmidt of $I-T$ in $\mathbb C\oplus \ell^2_s\oplus \ell^2_s$. To do so, we choose the following orthonormal basis of $\ell^2_s$: let $e_n=c_n \delta_n$,
i.e. the standard $\ell^2$ canonical basis properly rescaled (namely, we take $c_n=(1+n)^{-s}$). Then
\begin{align*}
\|I-T\|^2_{HS}&= \|(I-T)(1,0,0)\|^2+\sum_{n=1} \|(I-T)(0,e_n,0)\|^2+\sum_{n=1}\|(I-T)(0,0,e_n)\|^2\\
&=\sum_{n=1} \|(I-T)(0,e_n,0)\|^2+\sum_{n=1}\|(I-T)(0,0,e_n)\|^2\\
&\lesssim\sum_{n=1}^\infty|\langle (I-T)(0,e_n,0),(1,0,0)\rangle|^2 + \sum_{n=1}^\infty \sum_{k=1}^\infty |\langle (I-T)(0,e_n,0),(0,e_k,0)\rangle|^2+\\
&+\sum_{n=1}^\infty|\langle (I-T)(0,0,e_n),(1,0,0)\rangle|^2+\sum_{n=1}^\infty \sum_{k=1}^\infty |\langle (I-T)(0,0,e_n),(0,0,e_k)\rangle|^2
\end{align*}
Due to the symmetry $T^2(\bx,\by)=T^1(\by,\bx)$, it suffices to prove the desired bound on 
\begin{equation*}
\mathcal I:=\sum_{n=1}^\infty |\langle (I-T)(0,e_n,0),(1,0,0)\rangle|^2 + \sum_{n}\sum_{k}|\langle (I-T)(0,e_n,0),(0,e_k,0)\rangle|^2,
\end{equation*}
since the other can be proved similarly. By definition of $T$ we have
\begin{equation*}
\langle (I-T)x,y\rangle=\begin{cases}0&\text{ if }x=y=(1,0,0)\\
a_0(\sqrt k)-a_0(\sqrt{k+\varepsilon_k})&\text{ if }x=(0,e_k,0), y=(1,0,0)\\
a_n(\sqrt k)-a_n(\sqrt{k+\varepsilon_k})&\text{ if } x=(0,e_n,0), y=(0,e_k,0)
\end{cases}
\end{equation*}
Hence, taking into account the bounds of Theorem \ref{thm : RS2est} and the fact that $a_0(\sqrt k)=0$ if $k>0$
\begin{align*}
\mathcal I&\le \sum_k (1+k)^{2s} C|\varepsilon_k|^2e^{-K\sqrt{k+\varepsilon_n}}+
\sum_{n,k\ge  1}(1+n)^{-2s}(1+k)^{2s}|a_n(\sqrt{k+\varepsilon_n})-a_n(\sqrt k)|^2\\
&\le C_s\|\varepsilon_k\|_\infty +\sum_{n=1}^\infty (1+n)^{-2s}n^{\frac32}\log^6(1+n)
\sum_{k=1}^{Cn} (1+k)^{2s} \frac{|\varepsilon_k|^2}{k}e^{-\frac nk}\\
&
+\sum_{n=1}^\infty (1+n)^{-2s}n^{\frac32}\log^6(1+n)\sum_{k=Cn}^\infty  (1+k)^{2s} \frac{|\varepsilon_k|^2}{k}e^{-ck}.
\end{align*}
The last series can easily be bounded by $C_s\|\varepsilon_k\|_\infty$, thus if $s$ is sufficiently large so that $\sum_n (1+n)^{-2s}n^{\frac 32}\log^6(1+n)<\infty$,
we must only focus on 
$$\sum_{k=1}^{Cn} (1+k)^{2s} \frac{|\varepsilon_k|^2}{k}e^{-\frac nk}.$$
Recall that our assumptions imply
$|\varepsilon_k|\le \varepsilon k^{\alpha}$ for some $\alpha<-\frac54$; this in turn yields that 
\begin{align*}
\mathcal I&\le C_s\|\varepsilon_k\|+\sum_{n=1}^\infty (1+n)^{-2s}n^{\frac32}\log^6(1+n)
\sum_{k=1}^{Cn} (1+k)^{2s} \frac{|\varepsilon_k|^2}{k}e^{-\frac nk}\\
&\lesssim_s|\varepsilon|+|\varepsilon|\sum_{n=1}^\infty n^{-2s+\frac32}\log^6(1+n)
\sum_{k=1}^{Cn} k^{2s+2\alpha-1} e^{-\frac nk}.
\end{align*}
We rewrite thus the right hand side as
\begin{equation*}
\mathcal I
\lesssim |\varepsilon|+|\varepsilon| \sum_{n=1}^\infty n^{2\alpha+\frac32}\log^6(1+n)
\left(\frac1n\sum_{k=1}^{Cn} \left(\frac kn\right)^{2s+2\alpha-1} e^{-\frac nk}\right)
\end{equation*}
and we note that the inner sum is a Riemann sum, hence
\begin{equation*}\left(\frac1n\sum_{k=1}^{Cn} \left(\frac kn\right)^{2s+2\alpha-1} e^{-\frac nk}\right)\lesssim _s
\int_0^C x^{2s+2\alpha-1}e^{-x}\diff x,
\end{equation*} so provided $s\gg_{\alpha} 1$ we have
\begin{align*}
\mathcal I&\lesssim_{s,\alpha} |\varepsilon|+|\varepsilon| \sum_{n=1}^\infty n^{2\alpha+\frac32}\log^6(1+n) \int_0^Cx^{2s+2\alpha-1}e^{-x}\diff x\\
&\le C_{s,\alpha}|\varepsilon|,
\end{align*}
provided that $2\alpha+\frac32<-1$, which is indeed implied by our assumptions. Thus it follows that as $\varepsilon\to 0$ so does 
$\mathcal I \to 0$, proving that there exists a $\varepsilon(\alpha, s)$ such that $\|I-T\|_{\text{HS}}<1$. If we could make $\varepsilon$ independent
of $s$ we could deduce the theorem directly, but this seems challenging through the current methods. Hence, we employ Fredholm theory: namely, note that 
$$\|I-T\|_{\text{HS}}<+\infty \, \Rightarrow T \in \Phi(\C \oplus \ell^2_s \oplus\ell^2_s, \C \oplus \ell^2_s\oplus \ell^2_s),$$ 
that is, $T$ is Fredholm. Moreover, we have that 
$$i(T;\, \C\oplus \ell^2_s \oplus\ell^2_s, \, \C\oplus \ell^2_s\oplus \ell^2_s) = 0, \forall \, s > s_0.$$ 
This follows from the fact that 
$$I-T \in \text{HS}(\C \oplus \ell^2_s \oplus \ell^2_s,\C \oplus \ell^2_s\oplus \ell^2_s) \subset K(\C \oplus \ell^2_s \oplus \ell^2_s,\C \oplus \ell^2_s\oplus \ell^2_s).$$
A careful look at the proof proves that $T$ is actually $sc^+,$ since the HS norm is finite even from $\C \oplus \ell^2_s\oplus\ell^2_s$ into $\C \oplus \ell^2_{s'}\oplus\ell^2_{s'}$ for $s'>s,$ with the difference between $s$ and $s'$ bounded. Here, note that we have implicitly used the fact that the spaces $\ell^2_s$ are compactly embedded into one another, as previously mentioned in Section \ref{ssec : weights}

If we take $\varepsilon$ small enough (depending only on $\alpha$ now) so that there exists one $s_0$ such that 
$$\|I-T\|_{\text{HS}(\C \oplus \ell^2_{s_0} \oplus \ell^2_{s_0})}<1,$$
$T$ will be injective on $\C \oplus \ell^2_s \oplus\ell^2_s$ for all $s\ge s_0$, so that one can use Proposition \ref{thm : Fredholm}. Indeed, that result implies that $T$ is a Fredholm operator of index $0$ in the limit sequence space $\mathfrak s$, and since it is injective, it has to be bijective. Since it is also a continuous map, it is a linear isomorphism, and we have finished the proof, provided $\varepsilon_0=0$.

\vspace{2mm}

\noindent {\bf Step 2:} We now deal with the general case. In 
order to do it, we first apply {\bf Step 1} with the sequence 
$\tilde \varepsilon=(0,\varepsilon_1,\varepsilon_2,\dots)$. 

This allows us to obtain a new basis - with associated dual 
basis composed of dirac deltas and their composition with 
$\mathcal F$ - associated to the sampling points 
$\sqrt{n+\tilde \varepsilon_n}$. Let this new basis be denoted 
by $\{f_{1,n}\}_{n \in \N} \cup \{g_{1,n}\}_{n \in \N}$ (the 
ordering is irrelevant since the space is nuclear), and denote 
its dual basis by $(\delta_{x_n},\delta_{y_n}\circ 
\mathcal{F})_{n \ge 0}$. 

We now wish to employ Corollary \ref{cor : pertframe}. Indeed, define a ``perturbed'' dual basis as 
\begin{equation}\label{eq:new-frame} 
(\delta_{x_n},\delta_{y_n}\circ \mathcal{F})_{n \ge 1} \cup (\delta_{\varepsilon_0},\delta_{\varepsilon_0} \circ \mathcal{F}).
\end{equation} 
By Corollary \ref{cor : pertframe}, it suffices to find a seminorm $p_0$ such that, for any $f \in \mathcal{S}$, 
\[
 |h(0) - h(\varepsilon_0)|p(f_{1,0}) + |\widehat{h}(0) - \widehat{h}(\varepsilon_0)|p(g_{1,0}) \le C_p p_0(h).
\]
Hence, we simply take $p_0(h) = \|h'\|_{\infty} + \|\widehat{h}'\|_{\infty}$; as long as $|\varepsilon_0|$ is sufficiently small, the inequality above holds for $p=p_0$ with $C_{p_0} < 1$, and hence, by Corollary \ref{cor : pertframe}, we have that there is a ``new'' set of basis functions $\{f_n\}_{n \in \N} \cup \{g_n\}_{n \in \N}$ associated with the dual basis \eqref{eq:new-frame}. Hence, we finish also this case, and the Theorem is thus proved, as desired. 
\end{proof}

\subsection{Remarks} We gather below some relevant generalizations and comments on possible extensions and limitations of the proof of Theorem \ref{thm : pertSch}, starting with its version at \emph{all} Schwartz functions. 

\begin{remark}[Perturbed Schwartz formulas for general functions] Although we have restricted our attention so far to the more popular interpolation formula for even Schwartz functions from \cite{RV}, we note that a similar method may be used to obtain a result for \emph{general} Schwartz functions:  

\begin{theorem}\label{thm : pertSch-odd} Let $\{\varepsilon_n\}_{n \in \N}$ denote a sequence such that $|\varepsilon_n| < c (1+n)^{-\frac{5}{4} - \delta}$, for some $\delta > 0$. Then if $c$ is small enough, there exists a family of Schwartz functions $\{\tilde{h}_n\} \cup \{\tilde{k}_n\}$ such that for any $f \in \mathcal{S}_{\text{odd}}$, we have 
\[
f(x) = \sum_{n \ge 0} \left( f(\sqrt{n+\varepsilon_n}) \tilde{h}_n(x) + \widehat{f}(\sqrt{n+\varepsilon_n})\tilde{k}_n(x) \right).
\]
Moreover, the series converges \emph{absolutely} in the Schwartz topology, and such a decomposition is \emph{unique}. 
\end{theorem} 

\begin{proof}[Proof of Theorem \ref{thm : pertSch-odd}] We recall that in \cite{RV} the authors proved that there is a sequence of odd Schwartz functions $\{d_0^+\} \cup\{c_n\}_{n \ge 1}$ such that, for any $f \in \mathcal{S}_{\text{odd}}$, 
$$
f(x)=d_0^{+}(x) \frac{f^{\prime}(0)+i\widehat{f}^{\prime}(0)}{2}+\sum_{n=1}^{\infty} c_n(x) \frac{f(\sqrt{n})}{\sqrt{n}}-\sum_{n=1}^{\infty} \widehat{c_n}(x) \frac{\widehat{f}(\sqrt{n})}{\sqrt{n}},
$$
where the sum on the right-hand side converges absolutely for any $f \in \mathcal{S}_{\text{odd}}$. By using this formula in conjunction with Theorem \ref{thm : RV} and the canonical even-odd decomposition of any real function, we are able to write
\begin{align}
f(x) &= d_0^{+}(x) \frac{f^{\prime}(0)+i\widehat{f}^{\prime}(0)}{2} + a_0(f(0) + \widehat{f}(0)) + \frac{1}{2} \sum_{n\ge 1} f(\sqrt{n}) \left( \frac{c_n(x)}{\sqrt{n}} + a_n(x)\right) \cr 
    & + \frac{1}{2} \sum_{n \ge 1} f(-\sqrt{n}) \left( a_n(x) - \frac{c_n(x)}{\sqrt{n}}\right) + \frac{1}{2} \sum_{n \ge 1} \widehat{f}(\sqrt{n}) \left( \widehat{a_n}(x) - \frac{\widehat{c_n}(x)}{\sqrt{n}}\right)  \cr 
    & + \frac{1}{2} \sum_{n \ge 1} \widehat{f}(-\sqrt{n}) \left( \widehat{a_n}(x) + \frac{\widehat{c_n}(x)}{\sqrt{n}} \right). 
\end{align}

In analogy to the proof of Theorem \ref{thm : pertSch}, the first step is to suppose that the perturbations at hand satisfy $\varepsilon_0 = 0$. By doing that, a direct, careful redoing of the proof of Theorem \ref{thm : pertSch}, using, namely (see \cite[Eq.~(5.17)]{RS2}),
\begin{align*}
|c_n(x)|, |\widehat c_n(x)|&\lesssim n^{\frac34}\log^3(1+ n)\exp(-c'|x|/\sqrt n),\\
|c_n'(x)|, |{\widehat c}'_n(x)|&\lesssim n^{\frac54}\log^3(1+ n)\exp(-c'|x|/\sqrt n),
\end{align*}
yields that the following result holds: 

\begin{proposition}\label{prop : oddpert} Let $\{\varepsilon_n\}_{n \in \N}$ be a sequence such that $\varepsilon_0 = 0$ and $|\varepsilon_n| < c (1+n)^{-\frac{5}{4} - \delta}$. Then if $c$ is small enough, there exists a family of Schwartz functions $\{\tilde{f}_n\} \cup \{\tilde{g}_n\}$ such that for any $f \in \mathcal{S}_{\text{odd}}$, we have 
\[
f(x) = d_0^+ \frac{f'(0) + i \widehat{f}'(0)}{2} + \sum_{n \ge 1} \left( f(\sqrt{n+\varepsilon_n}) \tilde{f}_n(x) + \widehat{f}(\sqrt{n+\varepsilon_n})\tilde{g}_n(x) \right).
\]
Moreover, the series converges \emph{absolutely} in the Schwartz topology, and such a decomposition is \emph{unique}. 
\end{proposition} 

In order to be able to perturb the origin, besides the already-dealt-with case of $a_0(x) \cdot (f(0) + \widehat{f}(0))$ stemming from the even interpolation formula, another issue to be noticed is that the original interpolation formula for odd functions contains derivatives in $0$. More precisely, the dual perturbed formula given by Proposition \ref{prop : oddpert} contains the functional $\delta'_0+i\delta'_0\circ \mathcal F$. However, since
\begin{equation*}
\frac{\delta_\varepsilon-\delta_0}{\varepsilon}\to \delta'_0 \quad \text{ as } \varepsilon \to 0
\end{equation*}
in $\mathcal S'$, Corollary \ref{cor : pertframe} allows us to substitute those with a linear combination of standard Dirac deltas that are close 
enough to the origin: indeed, take $p_0(f) = \|f''\|_{\infty} + \|\widehat{f}''\|_{\infty}$, and note that, for $\{y_n\}$ given by the dual basis to the one induced by Proposition \ref{prop : oddpert}, and $\{\tilde{y}_n\}$ given by the \emph{same} basis, except for $n=0$, where one changes the element 
$$\frac{\delta_0' + i \delta_0' \circ \mathcal{F}}{2} \text{ by } \frac{1}{2} \left( \frac{\delta_{\varepsilon_0} - \delta_0}{\varepsilon_0} + i \frac{\delta_{\varepsilon_0} \circ \mathcal{F} - \delta_0 \circ \mathcal{F}}{\varepsilon_0} \right),$$ 
we have that (for an enumeration $\{x_n\}$ of the basis $\{\tilde{f}_n\} \cup \{\tilde{g}_n\}_n$ yielded by Proposition \ref{prop : oddpert})
\begin{align} \label{eq : bound-cor-der-sch}
\sum_{n \ge 0} |(y_n - \tilde{y}_n)(f)| p(x_n) & \le \left(\left| \frac{f(\varepsilon_0) - f(0)}{\varepsilon_0} - f'(\varepsilon_0)\right|  + \left| \frac{\widehat{f}(\varepsilon_0) - \widehat{f}(0)}{\varepsilon_0} - \widehat{f}(0)\right| \right)p(d_0^+) \cr 
 & \le C \varepsilon_0^2  \, p(d_0^+) p_0(f).
\end{align} 
In particular, for $p = p_0$, if $\varepsilon_0$ is small enough, then the constant on the right-hand side of \eqref{eq : bound-cor-der-sch} is smaller than $1$, allowing us to indeed employ Corollary \ref{cor : pertframe}, which then finishes the proof of Theorem \ref{thm : pertSch-odd}. 
\end{proof}

\end{remark}

\begin{remark}[Different perturbations in both space and frequency]\label{rmk : doublepert}
We also remark that, while the theorem is stated for one single perturbing sequence $\varepsilon_n$, a very similar argument also works
for when we apply two different perturbations, i.e.
\begin{equation*}
f\mapsto (f(\sqrt{n+\varepsilon_n}),\widehat f(\sqrt{n+\delta_n})).
\end{equation*}
This sampling will also induce a basis, provided one is careful when considering the issues arising with 
Poisson summation: namely, one obtains a basis only if one allows the interpolation formula to be written as
\begin{equation*}
    f(x)=(f(\varepsilon_0)+\widehat f(\delta_0))g_0(x)+\sum_{n=1}^\infty f(\sqrt{n+\varepsilon_n})g_n(x)+\sum_{n=1}^\infty \widehat f(\sqrt{n+\delta_n})h_n.
\end{equation*}
Trying to split the first term into two parts may break the uniqueness of 
decomposition due to non-classical Poisson formulas, in the same way as this happens in the Radchenko-Viazovska case (where it is clearer as $a_0=\widehat a_0$).
\end{remark}

\begin{remark}[Dependence on $s$]
As mentioned in the proof, to argue that we can take the limit $s\to \infty$ it would be enough to make $\varepsilon$ independent of $s$ or in other
words to find a bound $\mathcal I\lesssim |\varepsilon|$ which does not depend on $s$. While we can make the term
\begin{equation*}
\sum_{n=1}^\infty (1+n)^{-2s}n^{\frac32}\log^6(1+n)
\sum_{k=1}^{Cn} (1+k)^{2s} \frac{|\varepsilon_k|^2}{k}e^{-\frac nk}=:t
\end{equation*}
independent of $s$, the two bounds
\begin{align*}
\langle (I-T)(e_0),e_k\rangle&\le \sum (1+k)^{2s}|\varepsilon_k|^2 e^{-\frac{x+\varepsilon}{\sqrt \pi}}\\
\langle (I-T)(e_n),e_k\rangle&\le t+\sum_{n=1}^\infty (1+n)^{-2s}n^{\frac32}\log^6(1+n)\sum_{k=Cn}^\infty  (1+k)^{2s} \frac{|\varepsilon_k|^2}{k}
e^{-ck}
\end{align*}
are not, thus dooming this attempt. It does not seem unreasonable that a suitable renorming of $\ell^2_s$, combined with more efficient estimates, could make all the bounds independent of $s$, but it seems to be out of reach with current methods. 
\end{remark}

\begin{remark}[Schatten classes and Theorem \ref{thm : pertSch}] In order to prove Theorem \ref{thm : pertSch}, a crucial feature was the usage of the Hilbert-Schmidt norm of operators in order to deduce that, in particular, $I-T$ was compact in the right tower of Hilbert spaces. 

On the other hand, one readily notices that the Hilbert-Schmidt theory used was, as a matter of fact, not needed for our ultimate goal: indeed, we could instead in theory estimated a higher \emph{Schatten norm} 
$$\|I-T\|_{S_p}^p := \sum_{n\ge 0} |\lambda_n(I-T)|^p,$$
where $\{\lambda_k(S)\}_{k \ge 1}$ denotes the sequence of eigenvalues of the operator $S$ (where we omit the underlying Hilbert space where the operators are defined). The main issue with this strategy, however, is the lack of good estimates for the Schatten norm of general operators: indeed, to the best of our knowledge, the best known bound in general is given by the recent result of J. Delgado and M. Ruzhansky \cite{delgado2021schatten} (see also \cite{russo} for eariler related work). 

In particular, although in \cite[Corollary~5.4]{delgado2021schatten} sufficient conditions for an operator on spaces $\ell^2(\Z^n)$ to belong to the class $S_p$ are given, which when suitably modified for the weighted case give us sufficient conditions for $I-T$ to belong to the respective Schatten class, a careful inspection of that result shows that these are obtained \emph{in our particular case} as a direct interpolation between the Hilbert-Schmidt estimates given here, and the Schur test conditions employed in \cite{RS2} - in the same spirit as the work done in \cite{russo}. As such, it seems to us at the present moment that it is not possible to improve on the current range of estimates without improving on the Schur bound from \cite{RS2}. 
\end{remark}

\begin{remark}[A weaker version of Theorem \ref{thm : pertSch}]

We note that a more direct and elementary argument than the one used for the proof of Theorem \ref{thm : pertSch} is available, at the expense of a somewhat weaker conclusion. Nevertheless, since it helps contextualising {\bf Step 2} in the proof of Theorem \ref{thm : pertSch}, and since it is more striaghtforward than the previous one, we present it below: 

\begin{proposition}\label{prop : weakpert}
Let $F:\mathbb N\to \mathbb R$ be a function such that any polynomial $p=o(1/F)$. Then there exists
a constant $\delta$ depending on $F$ such that, for any sequence $(\varepsilon_n)$ such that $|\varepsilon_n|\le \delta|F(n)|$, there exists a family of Schwartz 
functions $\{f_n\}\cup \{g_n\}$ such that for any $f\in \mathcal S_{\text{even}}$
we have
\begin{equation*}
f(x)=\sum f(\sqrt{n+\varepsilon_n})f_n(x)+\widehat f(\sqrt{n+\varepsilon_n})g_n(x).
\end{equation*}
\end{proposition}
\begin{proof}
Since $\mathcal S$ is Montel and nuclear, it follows from the results of Section $2$ that 

$$A:=\{\delta_0+\delta_0\circ \mathcal F, \delta_1, \delta_1\circ \mathcal F,\dots\} = \{\alpha_n\}_{n \in \N}$$
is a Schauder basis for $\mathcal S'$. Let $T$ be the (a priori only densely defined) operator that sends
$\delta_{\sqrt n} \to \delta_{\sqrt{n+\varepsilon_n}}$, $\delta_{\sqrt n}\circ \mathcal F\mapsto \delta_{\sqrt{n+\varepsilon}}\circ \mathcal F$
and $\delta_{0}+\delta_0\circ \mathcal F\mapsto \delta_{\varepsilon_0}+\delta_{\varepsilon_0}\circ \mathcal F $. We claim that this operator
is continuous (hence it can be continuously extended to $\mathcal S'$) and invertible. To prove both, we shall employ a strategy similar to that of {\bf Step 2} in the proof of Theorem \ref{thm : pertSch}: it suffices to prove that there
exists $p_0$ a continuous seminorm on $\mathcal S$ such that for all other seminorms $p$ we have

\begin{equation}\label{eq: seminorms}
\sum_{n \in \N} p_0^*(\alpha_n-\beta_n) p(x_n)< C_p<\infty
\end{equation}
and that $C_{p_0}<1$, where $\{\alpha_n\}_{n \in \N}$ is the enumeration of $A$ above, $x_n$ is the induced enumeration of the Radchenko-Viazovska basis, and
$\beta_n=T(\alpha_n)$. 

To prove this, let $p_0(f):=\|f'\|_\infty+\|(\widehat f)'\|_\infty$. The series on the left-hand side of \eqref{eq: seminorms}  can then be bounded by
\begin{equation*}
S=|\varepsilon_0|+\sum_{n \neq 0} \frac{|\varepsilon_n|}{\sqrt n}p(a_n)+\sum_{n \neq 0} \frac{|\varepsilon_n|}{\sqrt n}p(\widehat a_n)=:I+II+III.
\end{equation*}
Indeed, if for instance $n$ is odd, 
\begin{align*}
p_0^*(\alpha_n - \beta_n) = \sup_{f \colon p_0(f) \le 1} |(\alpha_n - \beta_n)(f)| = \sup_{f \in \mathcal{S}\colon p_0(f) \le 1} |f(\sqrt{m}) - f(\sqrt{m+\varepsilon_m})| \le C \frac{|\varepsilon_m|}{\sqrt{m}},
\end{align*}
where $2m+1 = n.$ A similar argument works for $n$ even, but using the Fourier transform of $f$ instead. 

Clearly, it suffices now to prove the result for $p$ of the form $p_{\alpha,\beta}$, since these generate all seminorms of $\mathcal{S}$. Recall that, by Corollary \ref{cor : seminorms}, we have
\begin{equation*}
p_{\alpha,\beta}(a_n)\lesssim n^{\frac{\alpha+\beta}2}n^{\frac34}\log^3(1+n).
\end{equation*}
Thus, applying this latter inequality, it follows that
\begin{equation*}
II,III\lesssim_{\alpha,\beta} \sum_{n \ge 0} |\varepsilon_n| n^{\frac{\alpha+\beta}2}n^{\frac14}\log^3(1+n)\le C \cdot \delta \cdot K(F,\alpha,\beta).
\end{equation*}
Here, we used the rapid decay of $|\varepsilon_n|$ in order to offset the possible (at most polynomial) growth induced by the other factors. Choosing $\delta$ so that, for some $\alpha,\beta$ such that $p_0\le p_{\alpha,\beta}$ we have $CK<1$, this implies by  that $T$ is indeed an isomorphism, hence
$T(A)$ is a Schauder basis.
\end{proof}

It is also possible to prove this last result without resorting to Corollary \ref{cor : seminorms}: it suffices to notice that our estimates imply that $I-T$
maps an open set into a bounded set and by the fact that $\mathcal S$ is a Montel space, it follows that $I-T$ is compact, since $\mathcal S^*$ will be 
Montel as well. This implies
Fredholmness of $T$ and that $\text{ind}(T)=0$. Choosing $K$ so that $CK<1$ we have that $T$ is injective, hence an isomorphism.

\end{remark}

\section{Frames and Schwartz convergence of supercritical interpolation formulas}\label{sec : schw-int}
We recall once more that, in \cite{KNS}, the authors ask whether there exists a version of their interpolation formula that is valid in the topology of $\mathcal S$.
In this section, we first reframe the problem in one that is purely functional analytic and we then solve that problem for a particular 
class of sequences, whose prototypical case is that of small powers of integers. We will also see that, conditional on a conjecture of the second author and M. Sousa, the result can be significantly strengthened.

Finally, as an application, we show that a recent result of the second author with M. Sousa on \emph{discrete} Gaussian decay implying \emph{global analyticity} \cite{RS3} can be extended to the endpoint case under some technical conditions.

 We now set the groundwork for the next few results. We say that a pair of sequences $(\{x_n\}_n,\{y_n\}_n)$ is \emph{$(p,q)-$supercritical} if they satisfy 
\begin{equation}\label{eq : p-q-sep}
\max\left( \limsup_n |x_{n+1}|^{p-1} |x_{n+1} - x_n|, \limsup_n |y_{n+1}|^{q-1} |y_{n+1} - y_n| \right) < \frac{1}{2},
\end{equation} 
with $\frac{1}{p} + \frac{1}{q} = 1.$ Furthermore, we say that a sequence $\{x_n\}_{n\in\N}$ is $p-$\emph{separated} if there is $c>0$ such that 
\[
x_{n+1} - x_n \ge \frac{c}{\left(1+\min(|x_n|,|x_{n+1}|)\right)^{p-1}}.
\]
For a pair of $(p,q)$-supercritical sequences $(\{x_n\}_n,\{y_n\}_n)$, we say it is $(p,q)-$\emph{separated} if $\{x_n\}_n$ is $p-$separated and $\{y_n\}_n$ is $q-$separated. 

Using now the framework built in Section \ref{sec : prelim}, we know that the existence of frame expansion functions 
\footnote{the two conditions are equivalent only if the frame expansion functions are supposed to be particularly well behaved, i.e. if they are
supposed to form, say, a Bessel sequence, see the other equivalent conditions in \cite[Proposition~2.8]{Bonet2}}
is equivalent to the complementedness of the range of the analysis operator. In what follows, we prove exactly that this \emph{always} happens in case \eqref{eq : p-q-sep} holds, and hence the respective functionals always form a frame in that case. 

\begin{proposition}\label{prop : prelim-frame-KNS}
Given a pair $(\{x_n\}_n,\{y_n\}_n)$ of $(p,q)$-supercritical sequences, the functionals $(\delta_{x_n},\delta_{y_n}\circ \mathcal F)$ form a 
$\mathfrak s$-frame for $\mathcal S$.
\end{proposition}

\begin{proof}
Let $T:L^1\cap C^0\to \mathbb C^{\mathbb Z} \oplus \mathbb C^{\mathbb Z}$ be the map defined as
\begin{equation*}
T(f)=((f(x_n))_n, (\widehat f(y_n))_n).
\end{equation*}
We now recall \cite[Theorem~2]{KNS}: 

\begin{theoremA}\label{thm : KNS3} Let $(\{x_n\}_n,\{y_n\}_n)$ be a $(p,q)-$separated supercritical pair with parameters $p, q>1$, $\frac{1}{p}+\frac{1}{q}=1$, and let $s>0$ be such that $s \min (p, q) \geqslant 1$. Then there exist numbers $c, C>0$, such that, for every $f \in \mathcal{S}$,
$$
\begin{aligned}
& c\|f\|_{\mathcal{H}_{ p, q}^s}^2 \leqslant\left[\sum_{n \in \N }(1+|x_n|)^{(2 s-1) p+1}|f(x_n)|^2\right. \\
&\left.+\sum_{n \in \N}(1+|y_n|)^{(2 s-1) q+1}|\widehat{f}(y_n)|^2\right] \leqslant C\|f\|_{\mathcal{H}_{ p, q}^s}^2.
\end{aligned}
$$
\end{theoremA}
By Theorem \ref{thm : KNS3}, it follows then that 
$$T:\mathcal H^{s}_{p,q}\to \ell^2_{(2s-1)p +1,\{x_n\}}\oplus  \ell^2_{(2s-1)q + 1,\{y_n\}}$$
is a continuous embedding, where we define for a sequence $\{z_n\}$ and $r>0$ the space 
\[
\ell^2_{r, \{z_n\}} \coloneqq  \left\{ \{\by_n\}_n \colon \, \|\by_n\|_{\ell^2_{r,\{z_n\}}}^2 =\sum_{n \in \N} (1+ |z_n|)^r |\by_n|^2 < +\infty\right\}. 
\]
Since one may prove that 
\begin{equation*}
	\mathfrak s=\underset{\infty \leftarrow s}{\lim} \ell^2_{s,\{x_n\}}=\underset{\infty \leftarrow s}{\lim} \ell^2_{s,\{y_n\}}
\end{equation*}
and 
$\mathcal S=\lim \mathcal H^{s}_{p,q}$, it follows that $T$ is continuous from $\mathcal S$ to $\mathfrak s\oplus \mathfrak s$. Clearly, it is also injective, leaving only to prove
that $T(\mathcal S)$ is closed. In order to do this, we first claim that 
\begin{equation}\label{eq : T-eq}
	T(\mathcal S)=\bigcap_{s>0} T (\mathcal H^s_{p,q}). 
\end{equation}
Indeed, the inclusion $T(\mathcal{S}) \subset \bigcap_{s>0} T (\mathcal H^s_{p,q})$ follows at once by the fact that $\mathcal{S} = \bigcap_{s>0} \mathcal H^s_{p,q}.$ For the reverse inclusion, we use Theorem \ref{thm : KNS3} again: if a sequence $(\bx_n,\by_n) \in T(\mathcal{H}_{p,q}^s)$, then for each $s>0$ there is $f_s \in \mathcal{H}_{p,q}^s$ such that $f_s(x_n) = \bx_n, f_s(y_n) = \by_n.$ As $f_{s_1} - f_{s_2}$ vanishes at $\{x_n\}_n,$ and the Fourier transform of that function vanishes at $\{y_n\}_n$, it follows from the fact that $(\{x_n\}_n,\{y_n\}_n)$ is $(p,q)-$supercritical that $f_{s_1} \equiv f_{s_2},$ as long as $s_1,s_2$ are sufficiently large, such as for instance $s_1,s_2 > 2.$ Hence, there is $f \in \mathcal{S}$ with $f(x_n) = \bx_n, \widehat{f}(y_n) = \by_n$, as desired. 

Now, in order to conclude the proof, it is enough to note that $T(S)$ is \emph{closed} - to prove this, let $(\bx_n,\by_n)$
    be a a sequence in $T(\mathcal{S})$ whose limit
    $(\bx,\by)$ belongs to $\mathfrak s\oplus \mathfrak s$. It follows that $\{(\bx_n,\by_n)\}\subset T(\mathcal H^s_{p,q})$ for any $s>0$, and
    since those sets are closed for any fixed $s > 0$, $(\bx,\by)\in T(\mathcal H^s_{p,q}), \, \forall \, s > 0,$ and thus by the argument from above, so $(\bx,\by)\in T(\mathcal{S})$, as desired. 
\end{proof}

The following result then provides a functional-analytic characterization of Schwartz convergence of interpolation formulas:

\begin{corollary}\label{cor : complstr}
Let $\{x_n\}_n, \{y_n\}_n$ be a pair of $(p,q)$-supercritical sequences. Then there exist two families of functions $\{h_n\}_n, \{g_n\}_n$ such that
\begin{equation*}
f(x)=\sum f(x_n)h_n(x)+\sum \widehat f(y_n)g_n(x)
\end{equation*}
in $\mathcal S$ and such that the sum
\begin{equation*}
\sum \left(\bx_k h_k(x)+\by_k g_k(x)\right)
\end{equation*}
converges absolutely in $\mathcal S$ for each $(\bx_k, \by_k)$ in $\mathfrak s\oplus \mathfrak s$ if and only if the set
\begin{equation*}
\{(f(x_n),\widehat f(y_n)): f\in\mathcal S\}
\end{equation*}
is complemented in $\mathfrak s\oplus \mathfrak s$.
\end{corollary}

\begin{proof}[Proof of Corollary \ref{cor : complstr}] By Proposition  \ref{prop : strfr}, it follows that there exists a pair of sequences $(\{h_n\}_n, \{g_n\}_n)$ such that 
\[
\left( (\{h_n\}_n,\{g_n\}_n), (\delta_{x_n},\delta_{y_n} \circ \mathcal{F}) \right)
\]
is a Schauder frame, and such that the sum
\begin{equation*}
\sum \left(\bx_k h_k(x)+\by_k g_k(x)\right)
\end{equation*}
converges in $\mathcal S$ for each $(\bx_k, \by_k)$ in $\mathfrak s\oplus \mathfrak s$ if, and only if, the pair $(\delta_{x_n},\delta_{y_n} \circ \mathcal{F})$ is a $\mathfrak{s} \oplus \mathfrak{s}$-strong frame. By definition, this means that the image of $\mathcal{S}$ under the action of the sequence $(\delta_{x_n},\delta_{y_n} \circ \mathcal{F})$ is complemented. Since this latter image equals \emph{exactly} 
$\{(f(x_n),\widehat f(y_n)): f\in\mathcal S\}$, this finishes the proof.
\end{proof}

\begin{proof}[Proof of Theorem \ref{thm : KNSsoleven}] Since we have proved in Proposition \ref{prop : prelim-frame-KNS} that, whenever the sequences $\{x_n\}_n,\{y_n\}_n$ satisfy \eqref{eq : p-q-sep}, the functionals $(\delta_{x_n},\delta_{y_n} \circ \mathcal{F})$ form a $\mathfrak{s}-$frame for $\mathcal{S},$ by Corollary \ref{cor : complstr} we only need to prove that the frame in question is a strong $\mathfrak{s}-$frame. 

While proving that the range of the analysis operator is complemented is out of reach for general separated sequences, we are able to prove it for power sequences of small exponents: the technique will be to prove that we can select 
subsequences $x_{n_k}, y_{n_k}$ which induce a strong subframe, which in turn is achieved by employing the results in \cite[Section~5]{RS2}. 

\begin{proposition}\label{prop : preKNSsol}
Let $\{x_n\}_n, \{y_n\}_n$ satisfy the hypotheses of Theorem \ref{thm : KNSsoleven}. Then there exist a subsequence $x_{n_k}$, such that the functionals $(\delta_{x_{n_k}}, \delta_{y_{n_k}}\circ \mathcal 
F)$ form a strong $\mathfrak s-$frame for $\mathcal S_{\text{\emph{even}}}$.
\end{proposition}
\begin{proof}
Thanks to the isomorphism $\mathfrak s\oplus \mathfrak s \simeq \mathfrak s$, it suffices to prove that the map 
\begin{align*}
\mathcal R&:\mathcal S\to \mathfrak s\oplus \mathfrak s;\cr 
f&\mapsto ((f(x_k))_k,(\widehat f(y_k))_k)
\end{align*} is injective and that it has complemented range. Indeed, we prove that the range has finite codimension and is closed.
To prove this, recall that the operator $\tilde{\mathcal R}$ defined via
\begin{align*}
\tilde {\mathcal R}& : \mathcal S\to \C \oplus \C \oplus \mathfrak s\oplus \mathfrak s; \cr 
f & \mapsto (f'(0) + i \widehat{f}'(0),f(0) + \widehat{f}(0),(f(\pm\sqrt n)_n), (\widehat f(\pm\sqrt n)_n))
\end{align*}
is an injective Fredholm operator. We claim that there exists an operator 
$$\mathcal T:\C \oplus \C \oplus \mathfrak s\oplus \mathfrak s\to \mathfrak s\oplus \mathfrak s$$
such that $\mathcal R=\mathcal T\circ \tilde{\mathcal R}$ and that $\mathcal T$ is Fredholm and injective, concluding the proof (since the range of a Fredholm operator is automatically complemented).

In order to do so, first notice that, since $\C \oplus \C \oplus \mathfrak{s} \oplus \mathfrak{s} \simeq \mathfrak{s} \oplus \mathfrak{s}$, we shall assume without loss of generality that the operator $\mathcal{T}$ is an endomorphism on $\mathfrak{s} \oplus \mathfrak{s}$ instead. 

Such an operator has already been constructed in \cite[Section~5.B]{RS2}, with the crucial difference that the authors did not manage to prove the fact that the operator
is Fredholm in $\mathcal S$; we follow their construction here and we prove that $\text{ind}_{\ell_s^2\oplus \ell_s^2}(\mathcal{T})$ is indeed
uniformly bounded in $\mathfrak s$, which will then allow us to apply Theorem \ref{thm : Fredholm} to conclude that the operator is indeed Fredholm on $
\mathfrak s$ as well. We shall only sketch the proof of injectivity of $\mathcal T$, since clearly the same proof as in \cite{RS2} works here.

The main ingredient of the proof is the observation that, for $n\ge n_0$, there exists
$m_1^{\pm}(n) \in \Z$ such that $x_{m_1^{\pm}(n)}=\pm\sqrt{n+\varepsilon_n(m_1^{\pm})}$ and $\varepsilon_n(m_1^{\pm})$ is small enough to satisfy the hypotheses of Theorem \ref{thm : pertSch}. 

Indeed, let us deal with the case of positive sign ($+$), the other one being entirely analogous; in doing so, we shall omit the superscript. By the hypotheses on the sequence $\{x_n\}_{n \in \N},$ we have that, for $n \ge n_0$ sufficiently large, $\sqrt{n}$ belongs to an interval 
\[
\sqrt{n} \in [x_{i(n)},x_{i(n)+1}),
\]
where we have $|x_{i(n)+1} - x_{i(n)}| \le C |x_{i(n)}|^{-D} \le C n^{-D/2}$. Hence, setting $m_1(n)$ to be $i(n)$, we have that 
\[
|x_{m_1(n)}-\sqrt{n}| \le C \cdot n^{-D/2},
\]
and thus if $x_{m_1(n)} = \sqrt{n + \varepsilon_n(m_1)},$ we must have $|\varepsilon_n(m_1)| \le C \cdot n^{\frac{1-D}{2}}$, which satisfies the hypotheses of Theorem \ref{thm : pertSch} as long as $\frac{D-1}{2} > \frac{5}{4}$, that is, $D > \frac{7}{2}$. By the same token, we find that there exists $n_1 \in \N$ such that for $n \ge n_1,$ there is $m_2^{\pm} (n)$ such that $y_{m_2^{\pm}(n)} = \pm \sqrt{n + \tilde{\varepsilon}_n(m_2^{\pm})},$ where $\tilde{\varepsilon}_n(m_2^{\pm})$ satisfies the hypotheses of Theorem \ref{thm : pertSch}. Let then $N = \max\{n_0,n_1\}.$ 

Given an element $(\mathbf{x},\mathbf{y})\in \mathfrak s\oplus \mathfrak s$, recall from Theorem \ref{thm : pertSch} that $f_{(\bx,\by)}\in\mathcal S$ denotes the function
\begin{equation*}
\sum \bx_i a_i(x)+\sum \by_i \widehat a_i(x),
\end{equation*}
which we will denote by $f$ in the following for simplicity, 
and define the operator $\mathcal T_\alpha:\ell^2_s\oplus \ell^2_s\to \ell^2_s\oplus \ell^2_s$ as
\begin{align*}
\mathcal T_\alpha(\bx,\by)_1&=\left(\mathbf{x}_0, f(x_{i_1}), 
\dots, f(x_{i_{2N}}), \{f(x_{m_1^{\pm}(k)})\}_{k \ge N+1} \right),\\
\mathcal T_\alpha(\bx,\by)_2&=
\left(\mathbf{y}_0, \widehat f(y_{j_1}), \dots, \widehat 
f(y_{j_{2N}}), \{\widehat{f}(y_{m_2^{\pm}(k)})\}_{k \ge N+1}\right),
\end{align*}
where $i_1,\dots,i_{2N}$ is a sequence of integers that do not 
appear in $m_1^{\pm}(n)$, and $j_1,\dots,j_{2N}$ are defined analogously with respect to $m_2^{\pm}(n).$ In order to prove that $\mathcal T_\alpha$ is Fredholm, define the operator $T_{N}$ as
\begin{align*}
T_{N}(\bx,\by)_1&=\left(\mathbf{x}_0, f(x_{i_1}), 
\dots, f(x_{i_{2N}}), \mathbf{x}_{N+1},\mathbf{x}_{N+2},\dots\right),\\
T_{N}(\bx,\by)_2&=\left(\mathbf{y}_0, \widehat f(y_{j_1}), \dots, \widehat 
f(y_{j_{2N}}), \mathbf{y}_{N+1},\mathbf{y}_{N+2},\dots\right).
\end{align*}
The operator $T_{N}$ is sc-Fredholm, since its 
difference from the operator 
\begin{equation*}
S(\bx,\by):=((0,\dots,0, \bx_{N+1},\dots),
(0,\dots,0, \by_{N+1},\dots))
\end{equation*}
is a finite rank operator, which is continuous since 
$\|f_{(\bx,\by)}\|_{L_\infty}\le K(\|(\bx,\by)\|)$.
Now, the same proof as in Theorem \ref{thm : pertSch} implies 
\begin{equation*}
\|T_{N}-\mathcal T_{\alpha}\|_{\text{HS}(\ell^2_s\oplus 
\ell^2_s,\ell^2_{s+\varepsilon}\oplus 
\ell^2_{s+\varepsilon})}\le K,
\end{equation*}
where $\varepsilon>0$ is sufficiently small. Hence it follows that $T_{N}-\mathcal T_{\alpha}$,
is an $sc^+$-operator, so by Theorem
\ref{thm : Fredholm}, we have that $\mathcal T_\alpha$ is sc-Fredholm on $\mathfrak s\oplus \mathfrak s$.
To prove that $\mathcal T_\alpha$ is indeed injective 
provided $c_D>0$ is sufficiently small, let us first observe 
that it suffices to prove that $\mathcal T_\alpha$ is injective on $\ell_2^{10}\oplus \ell_2^{10}$. To prove this, we define the operator
\begin{align*}
\widehat T(\bx,\by)_1:=(\bx_0, f(t_1),\dots, f(t_{2N}), \bx_{N+1},\dots,),\\
\widehat T(\bx,\by)_2:=(\by_0, \widehat f(r_1),\dots, \widehat f(r_{2N}),\by_{N+1},\dots),
\end{align*}
where $t_1,\dots, t_{2N}$ is a sequence of real numbers all larger than $\sqrt{N}$ that do not belong to $\sqrt{\mathbb N}$, and $r_1,r_2,\dots,r_{2N}$ is defined similarly. Here, we may for instance take $\{t_i\}_i$ and $\{r_j\}_j$ to belong to, say $\sqrt{\N + \frac{1}{2}}$. 

Now, thanks to \cite[Lemma~5.4]{RS2}, it follows that $\widehat T$ is injective and since it is clearly Fredholm, it is an embedding, implying that there exists a constant $k$ such that
$\|\widehat T(\bx,\by)\|\ge k\|(\bx,\by)\|$. Then, it is only a matter of choosing $x_{i_j},y_{i_j}$ and $t_j$ so that their distance is small enough,
as that will imply $\|\widehat T-\mathcal T\|_{\mathcal L}<\varepsilon$ and if that is small enough, the triangle inequality implies
\begin{equation*}
\|\mathcal T(\bx,\by)\|\ge k'\|(\bx,\by)\|>0,
\end{equation*}
proving the injectivity. An admissible choice for, for instance, $x_{i_j}$ is simply the closest member of the sequence to $t_j$ - since the choice we made for $\{t_j\}_j$ above guarantees that, for such indices, the $\{x_{i_j}\}$ do \emph{not} coincide with any of the previously selected $\{x_{m^{\pm}(n)}\}_n$. If $c_D>0$ is small enough, we have
\begin{align*}
|f_{(\bx,\by)}(t_j)-f_{(\bx,\by)}(x_{i_j})|&\le \sum |\bx_n|a_n(t_j)-a_n(x_{i_j})|+\sum  |\by_n|\widehat a_n(t_j)-\widehat a_n(x_{i_j})|
\\&\lesssim \|t_j- x_{i_j} \|_{\ell_\infty} \sum (|\bx_n|+|\by_n |)\|a_n'\|_\infty \\&\lesssim \varepsilon(c_D,D) \|(\bx,\by)\|,
\end{align*}
where $\varepsilon(c_D,D)=\|t_j-x_{i_j}\|_{\ell_\infty}.$ By the same argument as in the beginning of the proof for $\varepsilon_n(m)$,
it follows that for $c_D\to 0$, $\varepsilon(c_D,\alpha)\to 0$. Since the same argument can be applied verbatim to $y_{i_j}$ and $r_j$, this proves the desired injectivity, which, by the considerations in the beginning of the proof, imply the desired injectivity. 
\end{proof}

As previously noted, it follows from Proposition \ref{prop : preKNSsol} and Corollary \ref{cor : complstr} that, under the conditions of Theorem \ref{thm : KNSsoleven}, the interpolation formulas obtained indeed converge in the Schwartz class, as desired. 
\end{proof}

\begin{remark} The results above point us to the following conjecture:
\begin{conjecture}\label{conj:KNS-frame}
Every $\mathfrak s-$frame for $\mathcal S$ induced by a pair of sequences as in Theorem 2 of \cite{KNS} is a strong frame.
\end{conjecture}
While our results fall 
short of what is needed to prove the conjecture, it is interesting to point out that improved estimates of $\|I-T_\varepsilon\|_{\mathcal L}$,
together with estimates on the decay of $\varepsilon$ needed for $T$ to be closed, would automatically improve our result: it is enough to first 
apply proposition \ref{prop : improvedSchwartzconditional} and then the same proof as in the previous theorem.

\end{remark}

\section{Analyticity of Perturbed Basis Functions}\label{sec : analytic}

\subsection{Proof of Theorem \ref{thm : entire}} We begin our discussion with a bound on the growth of the original basis function from \cite{RV} in \emph{compact} subsets of $\C$. 

In order to state our first preliminary result, we recall some notation from both \cite{RS2} and \cite{RV}: given the interpolation basis $\{a_n\}_{n\in\N}$ of Theorem \ref{thm : RV}, we define its associated \emph{Fourier $\pm1$-eigenfunctions} as 
\[
b_n^{\pm}(x) = \frac{a_n(x) \pm \widehat{a_n}(x)}{2}. 
\]
With that notation, we have:
\begin{proposition}\label{prop : growth-entire-basis}
For any compact subset $K$ of $\C$, we have
\begin{equation*}
    \|b_n^{\pm}\|_{L^\infty(K)}\le C^n \max_{z\in K} e^{\pi |z|^2},
\end{equation*}
where $C>0$ is an absolute constant independent of $K$. 
\end{proposition}

\begin{proof}
We prove the statement for the sequence $\{b_n^+\}_{n \in \N}$, and later indicate where the proof needs slight modifications to accommodate for the case of $\{b_n^{-}\}_{n \in \N}.$ \\

Recall now from \cite[Eq. (22)]{RV} that $b_n^+$ is defined as
\begin{equation}\label{eq : def-b_n}
    b_n^+(x):=\frac{1}{2} \int_{-1}^{1}g_n^+(z)e^{\pi i x^2 z}\diff z,
\end{equation}
where the integral is taken over a semicircle in 
the upper half plane joining $-1$ to $1$, and $g_n^+$ is a weakly holomorphic modular form of weight $3/2$, which is defined as
\begin{equation*}
    g_n(z):=\theta^3(z) \cdot P_n^+\left(\frac1J\right),
\end{equation*}
with the polynomial $P_n$ chosen so that
\begin{align*}
    g_n^+(z)=q^{-n/2}+O(q^{\frac12}).
\end{align*}
We now invoke \cite[Theorem~3]{RV}: this result states that 
\begin{align}\label{eq : generating-g_n}
\sum_{n \ge 0} g_n^+(z) e^{\pi i n \tau} & = \frac{\theta(\tau)(1-2\lambda(\tau))\theta^3(z) J(z)}{J(z) - J(\tau)}, \cr 
\sum_{n \ge 0} g_n^{-}(z) e^{\pi i n \tau} &= \frac{\theta(z)(1-2\lambda(z))\theta^3(\tau) J(\tau)}{J(z)-J(\tau)},
\end{align}
where the convergence holds, for instance, if $z \in \mathbb{S}^1 \cap \C_+$ and $\Im(\tau)$ is sufficiently large. The identities \eqref{eq : generating-g_n} will be instrumental in proving an upper bound for $g_n^{\pm}$. 

Indeed, we note that, upon integrating over the segment $[-1+iT,1+iT]$ in the complex plane, we have that 
\[
g_n^+(z) = \int_{-1+iT}^{1 + iT} \frac{\theta(\tau)(1-2\lambda(\tau))\theta^3(z) J(z)}{J(z) - J(\tau)} \cdot e^{-\pi i n \tau} \, \diff \tau.
\]
We now note that, for $\tau \in [-1 + iT,1+iT]$, if $T$ is sufficiently large, then there is $C>0$ such that 
\begin{equation}\label{eq : theta-lambda}
|\theta(\tau)| + |\lambda(\tau)| \le C.
\end{equation}
Indeed, both assertions may be proved at once by the respective $q-$ expansions of the functions $\theta, \lambda$, as highlighted in Subsection \ref{ssec : modular}. Now, one simply notes that 
\[
\frac{J(z)}{J(z) - J(\tau)} = \sum_{m \ge 0} J(\tau)^m J(z)^{-m},
\]
and, since the $q-$expansion of $J$ at infinity begins with $q^{1/2}$, we have that $|J(\tau)| \le C e^{-\pi \Im(\tau)}$ for $\Im(\tau)$ sufficiently large, and then we may bound 
\begin{align}\label{eq : upper-g-n-plus}
|g_n^+(z)| & \le C|\theta^3(z)| \int_{-1}^1  \left( \sum_{m \ge 0} |J(t + iT)|^m \left| \frac{1}{J(z)}\right|^{m} \right)  \cdot e^{\pi n T}\, \diff t \cr 
 & \le C \cdot e^{\pi n T} |\theta^3(z)| \int_{-1}^1 \left( \sum_{m \ge 0} e^{-\pi m T} \cdot (64)^m \right) \, \diff t \lesssim e^{\pi n T},
\end{align}
where in the last passage we have used that $\theta$ is bounded on $\mathbb{S}^1 \cap \C_+$. This implies, hence, by \eqref{eq : def-b_n}, that 
\[
\sup_{z \in K} |b_n^+(z)| \le C^n \sup_{z \in K} e^{ \pi |z|^2},
\]
where $C:= e^{\pi T}$, with $T$ chosen sufficiently large. 

In order to deduce the same result for $b_n^{-}$, we shall resort once more to \eqref{eq : generating-g_n}. Indeed, we may write:
\begin{equation}\label{eq : def-g_n-minus} 
g_n^{-}(z) = \int_{-1+iT}^{1 + iT} \frac{J(\tau)}{J(z)} \cdot  \frac{\theta(z)(1-2\lambda(z))\theta^3(\tau) J(z)}{J(z) - J(\tau)} \cdot e^{-\pi i n \tau} \, \diff \tau.
\end{equation} 
In order to analyse this new integral, one simply notes that the function $z \mapsto \frac{\theta(z) (1-2\lambda(z))}{J(z)}$ is \emph{bounded} on $\mathbb{S}^1 \cap \C_+$. The first step in order to see this is to use that $\theta$ is bounded on $\mathbb{S}^1 \cap \C_+$. Now, since 
\[
\lambda\left( -\frac{1}{z}\right) = 1 - \lambda(z), \, \lambda(z+1) = \frac{\lambda(z)}{\lambda(z)-1}, 
\]
it follows that 
\[
\lambda(1- 1/z) = 1 - \frac{1}{\lambda(z)} = q^{-1/2} + O(1).
\]
Now, since we also know from \eqref{eq : J-inv-dec} that
$$\frac{1}{J(1 - \frac{1}{z})} = -4096 q + O(q^2),$$
it follows that, for $z \in \mathbb{S}^1 \cap \C_+$, the function $z \mapsto \frac{1-2\lambda(z)}{J(z)}$ remains bounded (and moreover vanishes at $\pm 1$). Using this together with the estimates and techniques employed in the case of $g_n^+,$ we easily conclude from \eqref{eq : def-g_n-minus} that 
\[
|g_n^-(z)| \lesssim e^{\pi T n}, 
\]
which, in turn, implies the conclusion of the proposition for $b_n^{-}$, as desired.
\end{proof}

\begin{remark} As a direct consequence of the proof of Proposition \ref{prop : growth-entire-basis} above, if we let $d_n^{\pm}$ denote the Fourier $\pm1$-eigenfunctions 
\[
d_n^{\pm} = \frac{c_n \pm \widehat{c_n}}{2}
\]
associated with the interpolation basis $\{c_n\}_{n \in \N}$ of \emph{odd} Schwartz functions on the real line - which already figured above in the proof of Theorem \ref{thm : KNSsoleven} -, then we may equally conclude that, for each compact subset $K$ of $\C$, we have 
\[
\sup_{z \in K} |d_n^{\pm}(z)| \le C^n \sup_{z \in K} \left( |z| \cdot e^{\pi |z|^2}\right).
\]
Since the proof is almost identical to that of Proposition \ref{prop : growth-entire-basis}, with very few points where modifications are needed, we shall omit it. 
\end{remark}

\begin{remark} One might wonder whether the proof of Proposition \ref{prop : growth-entire-basis} above could have been accomplished by a better estimate on the basis functions than that of Theorem \ref{thm : RS2est}: indeed, a version of the Paley--Wiener theorem states that, if $|g(x)| \le C e^{-c|x|^2}$, then $\widehat{g}$ is of order $2$, with an estimate analogous to that of Proposition \ref{prop : growth-entire-basis}. 

Unfortunately, as the next result shows, the best order of decay possible for the functions $\{b_n^{\pm}\}_{n \in \N}$ is \emph{exponential}, which merely guarantees that the basis functions themselves belong to the class of analytic functions on a strip: 

\begin{proposition}\label{prop : bound-decay-b-n} For each $n \in \N$, there are $c_n, C_n> 0$ such that, whenever $|x| \ge C_n$, 
\[
|b_n^{\pm}(x)| \ge c_n \cdot \text{\emph{dist}}(x,\pm\sqrt{\N}) e^{-c|x|}. 
\]
\end{proposition}

\begin{proof} We first note that, from the Fourier expansion of $\{g_n^{\pm}\}_{n \in \N}$ at infinity and the definition of $b_n^{\pm}$, it follows from a change of contours that, whenever $|x| \ge \sqrt{n}$, it holds that 
\begin{equation}\label{eq : b-n-laplace} 
b_n^{\pm}(x) = \sin(\pi x^2) \int_0^{\infty} g_n^{\pm}(1+it) \, e^{-\pi x^2 t}\, \diff t. 
\end{equation} 
The next step is to use the definition of the functions $g_n^{\pm}$ in order to estimate the integral in \eqref{eq : b-n-laplace}. Note first that, for $|t|$ large, we may simply use the Fourier expansion of $g_n^{\pm}$ at infinity, which predicts that at $i \infty$
\begin{equation}\label{eq : g_n-decay-weak} 
g_n^{\pm}(z) = q^{-n/2} + O(q^{-(n-1)/2}).
\end{equation} 
On the other hand, we recall that 
\begin{align*}
g_n^+(z) &= \theta^3(z) \cdot P_n^+(1/J(z)), \cr 
g_n^{-}(z) & = \theta^3(z) (1-2\lambda(z)) P_n^{-}(1/J(z)),
\end{align*}
where $\{P_n^{\pm}\}_{n \ge 0}$ are sequences of monic polynomials with $P_n^-(0) = 0.$ We now claim that $P_n^+(0) \neq 0$ and that $(P_n^-)'(0) \neq 0$. Indeed, suppose that $P_n^+(0) = 0$. Then
\[
|J(z)^n \cdot g_n^+(z)| \le C |\theta^3(z)| |J(z)|\le C e^{-\pi \Im(z)},
\]
while, by \eqref{eq : g_n-decay-weak}, we have that the left-hand side above goes to a \emph{constant} as $\Im(z) \to \infty$, a contradiction. Hence, we conclude that $P_n^+(0) \neq 0$. By the same argument, we get that $(P_n^-)'(0) \neq 0$, as desired. 

Now, note that for $z = 1+it, \, |t| \text{ small}$, we have by the $q-$expansion of $J^{-1}(1-1/z)$ (see Subsection \ref{ssec : modular}) that 
\[
\left| \frac{1}{J(1+it)}\right| \le C e^{-\frac{2\pi}{t}}, 
\]
and hence 
\[
|g_n^+(1+it)| \ge |\theta^3(1+it)| \left( |P_n^+(0)| - C_n e^{-\frac{2\pi}{t}}\right).
\]
Now, using that, by the transformation properties of $\theta$, we have 
\[
\theta(1+it) \sim \frac{2}{\sqrt{t}} \exp\left(-\frac{\pi}{4t}\right) \quad \text{as } t \to 0^+,
\]
it follows that 
\begin{equation}
|g_n^+(1+it)| \ge \frac{c_n}{t^3} e^{-\frac{3\pi}{4t}}, \text{ for } t \text{ sufficiently small}. 
\end{equation}
Hence, using \eqref{eq : b-n-laplace} and \eqref{eq : upper-g-n-plus}, it follows that 
\begin{align*}
|b_n^+(x)| &\ge \text{dist}(x,\pm\sqrt{\N}) \left( c_n \int_0^{c_n}  t^{-3} e^{-\frac{3\pi}{4t}} e^{-\pi x^2 t} \, \diff t - C^n \int_{c_n}^{\infty} e^{-\pi x^2 t} \, \diff t \right) \cr 
          & \ge \text{dist}(x, \pm \sqrt{\N})  \left( c_n \int_0^{c_n} e^{-\frac{\pi}{2t}} e^{-\pi x^2 t}\, \diff t  - C_n e^{- c_n x^2}\right) \cr 
          & \ge \text{dist}(x, \pm \sqrt{\N}) \left( c_n' e^{-c|x|} - C_n e^{-c_n x^2} \right) \ge c_n \text{dist}(x,\pm \sqrt{\N}) \cdot e^{-c|x|}. 
\end{align*} 
This finishes the proof for $\{b_n^+\}$. For $\{b_n^-\}$, an entirely analogous argument is available, which allows us to finish also in that case, concluding the proof. 
\end{proof}
\end{remark} 


\begin{remark}[Sharpness of Proposition \ref{prop : growth-entire-basis}] As a last remark before proving Theorem \ref{thm : entire}, we note that the order of growth in Proposition \ref{prop : growth-entire-basis} cannot be much improved. Indeed, a natural question is whether the factor $C^n$ in the estimate of that result could be improved to a \emph{polynomial factor} - which would immediately, through an adaptation of the proof below, give an improvement of Theorem \ref{thm : entire}. 

As it turns out, a very quick argument is available to show that this \emph{cannot} happen. Indeed, suppose for the sake of a contradiction that there are $A,B > 0$ such that 
\[
|b_n^{\pm}(z)| \lesssim n^{A} e^{B|z|^2}, \quad \forall z\, \in \C. 
\]
Take now $f$ to be any compactly supported Schwartz function. Then, by virtue of the original Radchenko-Viazovska interpolation formula (Theorem \ref{thm : RV}) we would have that the sum 
\[
z \mapsto \sum_{n \ge 0} \left( f(\sqrt{n}) a_n(z) + \widehat{f}(\sqrt{n}) \widehat{a_n}(z)\right)
\]
converges \emph{uniformly on compact sets} of $\C$, and hence it has to converge to an entire function; on the other hand, by Theorem \ref{thm : RV}, it converges to $f$ for $z \in \R$. This yields an obvious contradiction, since then we would have necessarily that $f \equiv 0$. 

This concludes that \emph{no} sort of polynomial improvement may be achieved for Proposition \ref{prop : growth-entire-basis}. A similar argument, using now Gelfand-Shilov functions of suitable orders, shows that one cannot improve the exponent of $n$ in the exponential to anything strictly less than $1$ either, proving hence the effective optimality of the results of the proposition. 
\end{remark}

Given these considerations, we are now ready to prove Theorem \ref{thm : entire}.
\begin{proof}[Proof of Theorem \ref{thm : entire}]
We wish to prove that the following
interpolation formula holds:
\begin{equation*}
    f(x)=\sum f(\sqrt{n+\varepsilon_n})\tilde h_n(x)+\sum \widehat {f}(\sqrt{n+\varepsilon_n})\tilde{g}_n(x),
\end{equation*}
with $\tilde f_n, \tilde g_n$ 
entire of order $2$ and $\varepsilon_n$ a perturbing sequence decaying exponentially as 
described in the statement of the theorem.
Let us consider, as before,
the operator $T$, this time on a slightly different
weighted $\ell^2$ space, namely on $\ell^2(e^{cn}, \N)\oplus \ell^2(e^{cn}, \N)$ defined as
\begin{equation*}
    T(\bx, \by):=(\bx_0, (f_{\bx,\by}(\sqrt{n+\varepsilon_n}))_n, \by_0, 
    (\widehat f_{\bx,\by}(\sqrt{n+\varepsilon_n}))_n),
\end{equation*}
a priori only defined on compactly supported sequences. As in the case of 
$\ell_s^2( \N)$, a simple calculation proves that $\|I-T\|_{\text{HS}}<1$ which we only sketch here: using the same notation
as for $\ell_2(n^s, \N),$
we have
\begin{equation*}
    \mathcal I\le
    \sum_k \exp(ck) C |\varepsilon_k|^2 e^{-K\sqrt{k+\varepsilon_k}}
    +\sum_{n=1}^\infty\sum_{k=1}^\infty \exp(-cn)
    \exp(ck) n^{\frac32}\log^6(1+n)
    \frac{|\varepsilon_k|^2}k e^{-\frac nk}
\end{equation*}
thus if $|\varepsilon_n|\le c \exp(-Cn)$ with $c> 0$ sufficiently small and $C>0$ sufficiently large, the claim follows. We thus obtain that $T$ extends to the whole space and that it is invertible.
Using the inverse of $T$ we get that,
given $f \in \mathcal{S}_{\text{even}}$,
\begin{equation*}
    T^{-1}((f(\sqrt{n+\varepsilon_n}), \widehat f(\sqrt{n+\varepsilon_n}))
    =((f(\sqrt n)), (\widehat f(\sqrt n))).
\end{equation*}
Taking the $k-$th component of this, we obtain
\begin{equation*}
    f(\sqrt k)=\sum \gamma_{j,k}f(\sqrt{j+\varepsilon_j})+\tilde \gamma_{j,k}
    \widehat f(\sqrt{j+\varepsilon_j}),
\end{equation*}
where  
\begin{align*}
    \gamma_{n,k}&=T^{-1}((\delta_n),(0))_k\\
    \tilde{\gamma}_{n,k}&=T^{-1}((0),(\delta_n))_k
\end{align*}and $\delta_n$ denotes the sequence $(\delta_n(m))_m$. Since $T^{-1}$ is a bounded operator, it follows that
\begin{align*}
    |\gamma_{n,k}|&\le |\langle T^{-1}((\delta_n),(0)), ((\delta_k),0)\rangle|
    \\
    &\le e^{cn-ck} |\langle T^{-1}((e_n),(0)), ((e_k),0)\rangle|
    \\&\le \|T^{-1}\|
     e^{cn-ck},
\end{align*}
where $e_n$ is the normalised version of $\delta_n$. It follows that $|\gamma_{n,k}|, |\tilde{\gamma}_{n,k}|\lesssim e^{cn-ck}$. Applying this to the Radchenko-Viazovska interpolation formula 
together with Fubini-Tonelli implies
\begin{equation*}
    \tilde f_n(z)=\sum \left(\gamma_{n,k}a_k(z)+\tilde{\gamma}_{n,k}\widehat a_{n}(z)\right).
\end{equation*} 
Thanks to the bound in the previous
proposition, the right-hand side converges absolutely in compact subsets of $\C$: 
\begin{equation}\label{eq : bound-sum-analytic}
    \sum |\gamma_{n,k}|\|a_k(z)\|_{L^\infty(K)}\lesssim e^{cn} \sum e^{-ck}C^k\sup_{z \in K} \exp(\pi |z|^2)<\infty.
\end{equation}
This proves that the functions are entire (the same
argument, mutatis mutandis, works for $\tilde h_n$).
To prove our claim about the order, the proof
above needs to be refined slightly: instead
of considering the $L^\infty$ norm, we use a weighted
version of it; namely $\|f(z)e^{2\pi |z|^2}\|_{L^\infty(K)}$. By the same analysis as
above it follows
that the functions are of order $2$ and type $\pi$, as desired. 
\end{proof}

\begin{remark}
It is important to note that, in analogy to Theorem \ref{thm : pertSch}, Theorem \ref{thm : entire} also has a counterpart whose proof is much quicker and more elementary, as long as we are satisfied with a weaker conclusion. This is the content of the next result, which states that, under the same condition on the perturbations as in Theorem \ref{thm : entire}, our basis functions are still \emph{analytic on a strip} - as a matter of fact, we may still control the growth of their seminorms precisely:

\begin{theorem}\label{thm: exponential-decay-perturb}
Let $\{\varepsilon_n\}_n$ be a sequence of values and assume that $|\varepsilon_n|\le c e^{-Kn} $. Then there exists
a constant $K_0$ such that if $K>K_0$ and $c>0$ is sufficiently small depending on $K$, then we have
that the perturbed interpolating functions (whose existence follows from Theorem \ref{thm : pertSch}) belong to $S^1_{1,h}$ for some $h>0$, where the spaces $S^1_{1,h}$ are defined as in the proof of Corollary \ref{cor : seminorms}. 

As a consequence, the perturbed interpolating functions have to be \emph{analytic} in a strip. 
\end{theorem}

\begin{proof} Recall that, 
thanks to Theorem \ref{thm : RS2est}, all the Radchenko-Viazovska functions belong to $X:=S_{1,h}^1$ for some absolute $h$; moreover, we have
\begin{equation*}
\|a_n\|_X\lesssim n^{\frac34}\log^3(1+n) e^{\tilde hn},
\end{equation*}
for some $\tilde h>h$. Then we can define the
operator $T$ on $X$ as
\begin{equation*}
T(f)=\sum f(\sqrt{n+\varepsilon_n})a_n(x)+\sum \widehat f(\sqrt{n+\varepsilon_n})\widehat a_n(x).
\end{equation*}
A priori, the operator is only well defined as an operator from $X$ to $\mathcal S$. We shall prove that $\|(I-T)f\|_X< (1-\delta) \|f\|_X$, for some $\delta > 0,$ thus 
proving that $T$ is an isomorphism. To do so, we first note that the operator $f\mapsto f'$ is continuous on $X$ (see \cite[Chapter~5.1]{Gelfand})
\footnote{A slightly weaker conclusion can be obtained by combining the results in \cite{RealPW} and Grothendieck's factorisation 
theorem, which proves the existence of $\widehat h(h)$ such that $\frac{d}{\diff x}: S^1_{1,h}\to S^{1}_{1,\widehat h}$ is 
continuous.}. 
Thus, assuming $\|f\|_X\le 1$, we have
\begin{align*}
\|Tf-f\|_X&\le \sum |f(\sqrt{n+\varepsilon_n})-f(\sqrt n)|\|a_n(x)\|+\sum  |\widehat f(\sqrt{n+\varepsilon_n})-\widehat f(\sqrt{n})|\|\widehat a_n(x)\|\\
&\lesssim \sum \frac{|\varepsilon_n|}{\sqrt n}n^{\frac34}\log^{3}(1+n)e^{\tilde hn}\\
&\lesssim \sum n^{\frac14}\log^{ 3}(1+n)e^{\tilde hn}e^{-Kn}c.
\end{align*}
Thus if $K>K_0:=\tilde h$ and $c$ is sufficiently small, we see that $T$ is continuous and invertible.
Now let $f\in S_{1,h'}^1$ with $h'$ sufficiently larger than $h$. Since $T$ is invertible in $X$, we
know that
\begin{equation*}
f=T^{-1}\left(\sum f(\sqrt{n+\varepsilon_n})a_n(x)+\sum \widehat f(\sqrt{n+\varepsilon_n})\widehat a_n(x)\right)
\end{equation*}
Since $h'$ is sufficiently large, we can assume that the series in the right-hand side converges absolutely in $X$, hence we can switch $T^{-1}$ and summation to obtain
\begin{align*}
f&=\sum f(\sqrt{n+\varepsilon_n})T^{-1}(a_n(x))+\sum \widehat f(\sqrt{n+\varepsilon_n})T^{-1}(\widehat a_n(x))\\&=
\sum f(\sqrt{n+\varepsilon_n})h_n(x)+\sum \widehat f(\sqrt{n+\varepsilon_n})g_n(x),
\end{align*}
where $h_n:=T^{-1}(a_n)$ and $g_n:=T^{-1}(\widehat a_n)$. By boundedness of $T^{-1}$ and Theorem \ref{thm : RS2est}, we have
\begin{equation*}
\|h_n(x)\|_X\lesssim n^{\frac34}\log^{\frac 32}(1+n)e^{\frac{h}{4}n},
\end{equation*}
and similarly for $g_n$, which implies that for any $f\in\mathcal S$ the expression
\begin{equation*}
\mathbb{I}(f):=\sum f(\sqrt{n+\varepsilon_n})h_n(x)+\sum \widehat f(\sqrt{n+\varepsilon_n})g_n(x)
\end{equation*}
converges absolutely in $\mathcal S$, thus $\mathbb{I}$ is a continuous operator which coincides with the identity on $S_{1,h'}^1$, a dense subset of $\mathcal 
S$, so $\mathbb{I}=\text{Id}$. This proves the desired regularity of the perturbed interpolating functions.
\end{proof}

\end{remark}
\begin{remark}
These theorems lead one to wonder whether the Radchenko-Viazovska basis is indeed a Schauder basis in $S_1^1$; this would, in some sense, be the best possible result, as Proposition \ref{prop : bound-decay-b-n} implies that one cannot hope for anything better than exponential decay of a \emph{fixed order} for $b_n^{\pm}$. Achieving this result would require more refined techniques, as the known decay estimates are insufficient to prove it. Therefore, this question lies outside the scope of this manuscript but remains a subject we wish to investigate in the future.
\end{remark}
\subsection{An application to global analyticity} We first recall the following result from \cite{RS3}: 

\begin{proposition}\label{prop : RS-gauss} Let $\{x_n\}_n$ and $\{y_n\}_n$ be two sequences satisfying \eqref{eq : p-q-sep} for $p=q=2$. Suppose that $f:\R \to \C$ is in $H^1 \cap \mathcal{F}(H^1)$ and satisfies that 
\[
|f(x_i)| \le e^{-a |x_i|^2}, \, \forall i, \quad |\widehat{f}(y_j)| \le e^{-a|y_j|^2}, \, \forall j.
\]
Then there is $c_a>0$ such that 
\[
|f(x)| + |\widehat{f}(x)| \le e^{-c_a|x|^2}. 
\]
\end{proposition}
In particular, the Gaussian decay asserted in Proposition \ref{prop : RS-gauss} directly implies that \emph{both $f$ and $\widehat{f}$ are entire of order $2$} (see \cite[Theorems~2.5.12~and~2.5.13]{Boas}). It is hence natural to wonder whether this fact is preserved under \emph{general} sequences, not necessarily satisfying the supercritical condition \eqref{eq : p-q-sep}. The following result answers this question positively: 

\begin{theorem}\label{thm : hardy-entire} Suppose that $\{x_n\}_{n \in \Z} = \{\text{sign}(n)\cdot \sqrt{|n|+\varepsilon_{n}}\}_{n \in \Z}, \{y_n\}_{n \in \Z} = \{ \text{sign}(n)\sqrt{|n|+\delta_n}\}_{n \in \Z},$ where 
\[
\sup_n e^{Cn}\left( |\varepsilon_n| + |\delta_n| \right) \le c,
\]
for $c > 0$ sufficiently small. Then, there is $\alpha_0 >0$ such that, if $f \in \mathcal{S}$ is such that 
\[
|f(x_i)| \le e^{-\alpha |x_i|^2}, \, \forall i, \quad |\widehat{f}(y_j)| \le e^{-\alpha|y_j|^2}, \, \forall j,
\]
with $\alpha \ge \alpha_0$, then $f$ and $\widehat{f}$ are both \emph{entire} of order 2. 
\end{theorem}

\begin{proof} By Theorem \ref{thm : pertSch} - or, alternatively, by the main theorems in \cite{RS2} -, it follows that
\[
f(t) = \sum_{n \ge 0} \left( f(x_n) h_n(t) + \widehat{f}(y_n) g_n(t) \right)
\]
holds in the absolute sense for all $t \in \R$. Hence, if we define for $z \in \C$ fixed,
\begin{equation}\label{eq : complex-interpol} 
f(z) = \sum_{n \ge 0} \left( f(x_n) h_n(z) + \widehat{f}(y_n) g_n(z)\right),
\end{equation} 
it follows by Theorem \ref{thm : entire}, the assumptions of Theorem \ref{thm : hardy-entire} and \eqref{eq : bound-sum-analytic} that \eqref{eq : complex-interpol} actually converges absolutely for each $z \in \C$. Moreover, using \eqref{eq : bound-sum-analytic} once more shows that 
\[
\sup_{z \in K} |f(z)| \lesssim \sup_{z \in K}e^{\pi |z|^2} \sum_{n\ge 0} e^{(c-a)n} \lesssim  \sup_{z \in K} e^{\pi |z|^2},
\]
showing thus that $f$ is entire of order 2. Since the exact same argument may be applied to $\widehat{f}$, this concludes the proof of Theorem \ref{thm : hardy-entire}. 
\end{proof}

It is not clear from the proof above whether the restriction $\alpha \ge \alpha_0$ is necessary. The investigation of the necessity of such a condition seems like an interesting problem, which we wish to explore in future work - as well as, in case such a necessity is confirmed, the determination of the \emph{optimal} $\alpha_0$ for which Theorem \ref{thm : hardy-entire} holds. 

\section{Comments and remarks}\label{sec : remarks}

\subsection{Theorem \ref{thm : pertSch} and the endpoint case of Theorem \ref{thm : per}} As previously noted, Theorem \ref{thm : per} admits perturbations that are slightly more general than those in Theorem \ref{thm : pertSch}. Indeed, in \cite{RS2}, the authors also prove that, if $T_\varepsilon$ denotes the operator from Theorem \ref{thm : pertSch} associated to a sequence $\varepsilon:= \{\varepsilon_n\}_n$, then
$$\|I-T_\varepsilon\|_{\mathcal L}\le c_s \|\varepsilon\|,$$
 where 
 \begin{equation*}
\|\varepsilon\|:=\sup_n |\varepsilon_n|(1+n)^{\frac54}\log^3(e+n).
\end{equation*}
This allows one to obtain other perturbation formulas with a slightly weaker decay requirement, which is obtained at the expense
of convergence in $\ell_s^2$ rather than in $\mathfrak s$. 

Our technique for extending the result to the limit $s\to \infty$ does not work directly here, since having bounded operator norm does not imply compactness. However, we can still obtain some partial results in this direction:
\begin{proposition}\label{prop : improvedSchwartzconditional}
Assume there exists $c$ such that, for all $s$, for every perturbation $\{\varepsilon_n\}_{n=1}^\infty$ with $\|\varepsilon\|<c$ we have
that $T_\varepsilon:\ell_s^2\oplus \ell_s^2\to \ell_s^2\oplus \ell_s^2$ is a closed-range operator
and $T_\varepsilon :\mathfrak s\oplus \mathfrak s\to \mathfrak s\oplus \mathfrak s$ is as well. Then 
for any such perturbation $\varepsilon$ the conclusion of Theorem \ref{thm : pertSch} holds.
\end{proposition}
\begin{proof}
Let $\varphi:X\to \mathcal L(\ell_s^2\oplus \ell_s^2)$ be the map that sends a sequence $\varepsilon_1,\dots, \varepsilon_n,\dots$
to the associated perturbation operator $T_\varepsilon$, where $X:=\{(x_n):\sup |x_n|(1+n)^{\frac54}\log^3(e+n).<\infty\}$ is
the Banach space equipped with the natural norm. By the results of \cite[Section~4]{RS2} we know that $\varphi$ is a well-defined Lipschitz map (for $s$ large enough, where the Lipschitz constant depends on $s$). Moreover, we know that
there exists $c$ such that 
\begin{equation*}
\varphi(B_c(0))\subset \bigcap_{s\gg 1} \mathcal L_\text{inj}(\ell_s^2\oplus \ell_s^2).
\end{equation*}
This, together with our hypothesis of closed-range, implies
\begin{equation*}
\varphi(B_c(0))\subset \bigcap_{s\gg 1}  \Phi^+(\ell_s^2\oplus \ell_s^2).
\end{equation*}
Since for any Banach space $X$, $\partial \Phi\cap \Phi^{\pm}=\emptyset$ (\cite{muller}, lemma 18.1), it follows that
\begin{equation*}
\varphi(B_c(0))\subset \bigcap_{s\gg 1} \Phi(\ell_s^2\oplus \ell_s^2).
\end{equation*}
Since $B_c(0)$ is connected and $\varphi$ is continuous, it follows that $\text{ind}_{\ell_s^2\oplus \ell_s^2}(\varphi(B_c(0)))=\{0\}$.
Theorem \ref{thm : Fredholm} thus implies that for any $\varepsilon\in B_c(0)$, $T_\varepsilon$ is Fredholm of index $0$,
thus an isomorphism. The result follows.
\end{proof}

\subsection{Conjecture \ref{conj:KNS-frame} and topological aspects of Complementedness} One could also wonder whether there is a purely topological proof of the complementedness (or non-complementeddness) of the range:
indeed, this is not unheard of in analysis, as there are various topological proofs of analytical results that follow this route. Two examples are (see \cite[Theorem 5.18, Example 5.19]{Rudin})
\begin{itemize}
\item The boundedness of the Szegő  $Q:L^1\to L^1$ is equivalent to the complementedness of $H^1$ in $L^1$, 
which is impossible since $H^1$ does not have the Dunford-Pettis property.
\item Similarly, $Q:L^p\to L^p$ is bounded if and only if $H^p$ is complemented in $L^p$ (which is clearly true)\footnote{We sketch here a proof of the equivalence claimed:
clearly if the Szegő projection were continuous then $H^p$ would be complemented in $L^p$ as $Q$ would be the required projection. On the other hand, assume $H^p$ is complemented, and thus such a projection $P:L^p\to H^p$ exists. Then we can define another projection
$\tilde P:\tilde P(f)=\frac1{2\pi}\int_{\mathbb T}\tau_{-\vartheta}P(\tau_\vartheta(f))\diff \vartheta$,where $\tau_z$ is the right translation
operator. It follows that $\tilde P$ is rotationally invariant and one can prove that $\tilde P=Q$ must hold.}
\end{itemize}
Indeed, in \cite[Theorem~5.3,~Corollary~5.4]{FFS} it is claimed, without proof, that more generally the range of many $\mathfrak s-$ frames is
necessarily complemented, implying in particular our conjecture:
\begin{corollary}[\cite{FFS}]
Let $X_s$ be a family of Hilbert spaces such that 
\begin{align*}
\emptyset\neq \bigcap_{s\in\mathbb N} X_s\subset\dots \subset X_2\subset X_1\subset X_0\\
\|\cdot\|_0\le \|\cdot\|_1\le \|\cdot\|_2\le \dots\\
X:=\bigcap X_s\text{ is dense in each }X_s.
\end{align*}
Let $\{g_i\}\in (X')^{\mathbb N}$ and assume that for every $s$ there exit constant $0<A_s\le B_s<\infty$ such that for all $f\in X$
\begin{equation*}
A_s\|f\|_s^2\le \sum |i^s g_i(f)|^2\le B_s \|f\|_s^2.
\end{equation*}
Then the following are equivalent
\begin{enumerate}
\item There exists a family $\{f_i\}\in X^{\mathbb N}$ such that for every $s\in \mathbb N$ one has
\begin{equation*}
\exists 0<A_s'\le B_s'<\infty:\ \tilde A_s' \|g\|_{X_s'}^2\le \sum |i^{-s}g(f_i)|^2\le B_s' \|g\|_{X_s'}
\end{equation*}
and
\begin{align*}
\forall f\in X, f=\sum g_i(f)f_i\\
\forall g\in X', g=\sum g(f_i)g_i.
\end{align*}
Moreover, denoting by $g_i^s$ the extension of $g_i$ to $X_s$
\begin{align*}
\forall f\in X_s, f=\sum g_i^s(f)f_i\\
\forall g\in X_s', g=\sum g(f_i)g_i^s.
\end{align*}
\item The set $\{\{g_i(f)\}:f\in X\}$ is nuclear and has the properties $(DN)$ and $(\Omega)$.
\item The set $\{\{g_i(f)\}:f\in X\}$ is complemented in $\mathfrak s$.
\end{enumerate}
\end{corollary}
Unfortunately, their claim is mistaken. Indeed, 
 it is impossible to prove the complementedness from a purely topological point of view. More precisely, consider the following argument: by Borel's theorem \cite[Theorem~26.29]{Voigt}  there exists an exact sequence
\begin{equation*}
0\hookrightarrow s\hookrightarrow s\rightarrow \mathbb C^{\mathbb N}\rightarrow 0
\end{equation*}
which fails to split since $\mathbb C^{\mathbb N}$ lacks the $DN$ property (see \cite[Section~5.2]{Derived}). In other words, there is a subspace $X$ of 
$\mathfrak s$  that is isomorphic to it and yet is not complemented. 
This also provides a counterexample to \cite[Corollary~5.4]{FFS} (which, if true, would have solved our conjecture):
consider the map $g_i: \mathcal D([0,2\pi])\to  \mathfrak s, g_{2i}(f):=\widehat f(i), g_{2i-1}=\widehat f(-i)$. Then it follows from standard results
in Fourier series that for $f\in H^s_0([0,2\pi])$ we have
\begin{equation*}
c_s\|f\|_{H^s}^2\le \sum |g_i(f)i^s|^s\le C_s \|f\|_{H^s}.
\end{equation*}
Then it is easy to see that $G:= (g_0,g_1, g_{-1}, g_2,g_{-2}\dots)$ is an embedding of $\mathcal D$ into $\mathfrak s$ and as such $G(\mathcal D)\simeq \mathcal 
D\simeq \mathfrak s$. To see that $\mathcal D\simeq \mathfrak s$,
 notice that $f\mapsto f\circ \arctan$ is an isomorphism between $\mathcal D$ and $\mathcal S$ and the 
latter is isomorphic to a complemented subset of $\mathcal D$ so it has properties $DN$ and $(\Omega)$. Moreover, we note that $G$ is an isomorphism 
between $C^\infty_p([0,2\pi])$ (the set of smooth functions on $[0,2\pi]$ that admit a periodic smooth extension to $\mathbb R$) and $\mathfrak 
s$. Despite having
properties $DN$ and $(\Omega)$, $G(D)$ is not complemented in $\mathfrak s$: indeed, we claim that we have the following exact sequence
\begin{equation*}
0\rightarrow \mathcal D\overset{G}{\rightarrow } \mathfrak s\xrightarrow{B\circ G^{-1}} \mathbb C^{\mathbb N}\rightarrow 0,
\end{equation*}
where $B$ is the Borel map $B:C^\infty([0,2\pi])\to \mathbb C^N; f\mapsto f^{(n)}(0)$ and that it does not split (see the commutative diagram
below), which follows by the same argument as above. 
\begin{figure}[ht]
\center
\begin{tikzcd}
                  & \mathfrak s \arrow[r, hook]                        & \mathfrak s \arrow[d, "G^{-1}", bend left]                                                       &                                       &   \\
0 \arrow[r, hook] & \mathcal D \arrow[r, "\iota", hook] \arrow[u, "G"] & {C^\infty_p([0,2\pi])} \arrow[r, "\scriptstyle{f\mapsto f^{(n)}(0)}"', two heads] \arrow[u, "G", bend left] & \mathbb C^{\mathbb N} \arrow[r, hook] & 0
\end{tikzcd}
\end{figure}
In particular, $\{g_i\}$ is not a dual pre-$F$-frame, thus providing a counterexample to corollary $5.4$. In other words, despite the fact
that for each $s$, $(f\mapsto \widehat f(n))_{n\in\mathbb Z}$ is a frame for $H_0^s([0,2\pi])$ and the range $G(\mathcal D)$ is isomorphic
to a complemented subspace of $\mathcal S$, there is no family of functions $h_n\in \mathcal D$ such that
\begin{equation*}
f(x)=\sum_{n\in\mathbb Z} \widehat f(n) h_n(x)
\end{equation*}
and so that such an expansion also holds in each $H^s_0$. Heuristically, this is not surprising, as the natural choice of functions, $e^{inx}$,
clearly does not belong to $\mathcal D$.

\vspace{2mm}

While Conjecture \ref{conj:KNS-frame} seems to be out of reach with our current methods, we provide below a partial description of the range of 
the analysis operators associated to Kulikov-Nazarov-Sodin frames.

\begin{theorem}
Let $\{x_n\}, \{y_n\}$ be a pair or $(p,q)-$separated sequences and consider the map $\tilde T:\mathcal S\oplus \mathcal S\to \mathfrak s\oplus \mathfrak s$ defined as
\begin{equation*}
\tilde T(f,g)=(f(x_n), \widehat g(y_n)).
\end{equation*}
Then $\tilde T$ is surjective.
\end{theorem}
\begin{proof}
It suffices to prove that both $\tilde T_1^1: f\mapsto (f(x_n))$ and $T_2^2: g\mapsto \widehat g(y_n)$ are surjective. The result follows from the following
description of subalgebras of $s$:
\begin{lemma}[Theorem 4.8 of \cite{Tomasz1}; Proposition 4.3 of \cite{Tomasz2}]\label{lemma : subalg}
Let $X$ be a closed and $*$-closed subalgebra of $\mathfrak s$. Then there exist a decomposition of $\mathbb N$ in disjoint subsets $Z,A_0,$ $A_1,\dots, A_n$,$\dots, \in 
\mathcal P(\mathbb N)$ (where $Z$ may be empty)
such that $X=\overline{\text{span}}_{i}(\mathbbm 1_{A_i})$.
\end{lemma}
Since $T_1^1(\mathcal S)$ is a $*-$closed subalgebra of $\mathfrak{s}$, it follows from the lemma that it is the closed span of $\mathbbm 1_{A_i}$.
Moreover, since there is no $x_i$ such that all $f\in\mathcal S$ satisfy $f(x_i)=0$, $Z_0=\emptyset$.
We claim that, upon rearranging the $A_i$, we have $A_i=\{i\}$. In fact, it suffices to prove that $|A_i|\le 1$, since $\{A_i\}$ is a partition. 
For the sake of contradiction, assume that $A_i\supset \{n,m\}$. This would imply that for any $f\in\mathcal S$ we would have
$f(x_n)=f(x_m)$, which is easily contradicted by the function $e^{-|x-x_n|^2}$. The same argument implies $T_2^2(\mathcal S)=\mathfrak{s}$ as well.
\end{proof}

The only constraints on $T(\mathcal S)$ are then the ones due to the summation formulas induced by the crystalline measures $\mu$ 
whose support is contained in $\{x_n\}$ and such that the support of $\widehat \mu$ is contained in $\{y_n\}$, indicating that a
possible avenue to prove complementedness is through the  study of the set of such crystalline measures - see, for instance, \cite{goncalves}. 

This set is known to be finite-dimensional in the subcritical case and infinite dimensional in the supercritical case thanks to the work of Kulikov, Nazarov and Sodin (see \cite[Section~7.3]{KNS}).

\begin{remark}
Since clearly $T(\mathcal S)=\tilde T(\Delta)$ and $\Delta=\{(f,g):f=g\}$ is complemented in $\mathcal S\oplus \mathcal S$, one
may wonder if this is enough to imply the complementedness of $T(\mathcal S)$. We provide here an example of why this is not the case:
let $X$ be a hereditarily indecomposable separable Banach space (see \cite{Gowers} for the construction of one). Recall that since $X$ is separable, there exists a surjection $T:\ell^1\to X$ such that
the induced map $S:\ell^1/\text{ker}(T)\to X$ is an isometry.
Then since $X$ is hereditarily indecomposable, no infinite dimensional subspace with infinite codimension can be complemented, so it suffices to 
choose $Y$ complemented in $\ell^1$ with infinite dimension and codimension.\par
 Even assuming that $T(\mathcal S)$ to be closed is not enough: let $X=\ell_\infty\oplus c_0$, $Y=\ell_\infty$,  $Z=\{0\}\oplus c_0$
and let $T:X\to Y$ be defined as $T(x,y)=x+y$. 
Then $T$ is a surjective map, $Z$ is complemented  and $T(Z)$ is a closed subset of $Y$ yet it fails to be complemented (since $c_0$ is not complemented
in $\ell_\infty$; see \cite[theorem~2.a.7]{Lindenstrauss})
\end{remark}

\printbibliography[heading=bibintoc,title=References]

\end{document}